\documentclass[10pt,noinfoline]{imsart}

\usepackage[colorlinks,citecolor=blue,urlcolor=blue]{hyperref}
\usepackage{palatino}
\usepackage{bm}
\usepackage{geometry}
\usepackage{graphics,epsfig,rotate,lscape,graphicx,amsmath,amsthm,amssymb,float,amsfonts,amsbsy,delarray,sectsty,amsfonts,amscd,pifont}
\usepackage[mathscr]{euscript}  
\usepackage{multirow}
\usepackage{soul,float}
\usepackage{mathrsfs}
\usepackage{epsf}
\usepackage{psfrag}
\usepackage{fancyvrb}
\usepackage{enumerate}
\usepackage{subfigure}
\usepackage{sidecap}
\usepackage{framed}
\usepackage[usenames,dvipsnames]{color}
\usepackage[numbers]{natbib}
\definecolor{shadecolor}{gray}{0.90}
\newtheorem{claim}{Claim}

\newtheorem{theorem}{Theorem}

\newtheorem{corollary}[theorem]{Corollary}
\newtheorem{definition}{Definition}
\newtheorem{assumption}{Assumption}

\newcommand{\bs}{\boldsymbol}

\begin{document}

\begin{frontmatter}

\title{A generalized quadratic estimate for random field nonstationarity
\thanksref{code}}
\runtitle{Generalized quadratic estimate}

\begin{aug}
  \author{\fnms{Ethan}  \snm{Anderes}\thanksref{a, t1}\ead[label=e1]{anderes@ucdavis.edu}}
    \and
  \author{\fnms{Joe} \snm{Guinness}\thanksref{b}\ead[label=e2]{jsguinne@ncsu.edu}}

  \runauthor{Anderes and Guinness}

  \address[a]{Department of Statistics, University of California, Davis CA 95616, USA. \printead{e1}}

  \address[b]{Department of Statistics, North Carolina State University. \printead{e2}}

  \thankstext{t1}{EA was partially supported by NSF CAREER grant DMS-1252795}
  \thankstext{code}{All the source code used for the simulations and graphics in this paper is publicly available through the on-line repository \url{https://github.com/EthanAnderes/NonstationaryPhase.jl}}

\end{aug}

\begin{abstract}
In this paper, we attempt to shed light on a new class of nonstationary random fields which exhibit, what we call, {\em local invariant nonstationarity}. We argue that the local invariant property has a special interaction with a new  generalized quadratic estimate---also derived here---which extends an estimate originally developed for gravitational lensing of the Cosmic Microwave Background in Cosmology \cite{hu2001mapping, hu2002mass}. The nature of this interaction not only encourages low estimation bias but also enables accurate (and fast) quantification of Frequentist mean square error quantification of the estimated nonstationarity. These quadratic estimates are interesting, in their own right, as they detect and estimate nonstationarity by probing correlation among Fourier frequencies, the absence of which is the characterizing feature of weak stationarity (by Bochner's Theorem). Moreover, this generalized quadratic estimate can be computed with a Fourier characterization that runs in $\mathcal O(n\log n)$ time when observing the field on a uniform grid of size $n$ in $\Bbb R^d $. Finally, the work presented here partially addresses two other problems associated with the statistical theory of nonstationarity: 1) estimating the phase of a spatially varying modulated stationary random field and 2) identifying a larger class of nonstationary random fields which admit an extension of the quadratic estimator of gravitational lensing that extends the same attractive statistical properties.
\end{abstract}
\end{frontmatter}

\tableofcontents

\section{Introduction}

Many data sets in time series and spatial statistics show clear signs of nonstationarity \cite{sampson2010constructions,fuglstad2015does}.
Despite the prevalence of nonstationary data, theory for understanding and estimating nonstationary random field models is still underdeveloped compared to what is known for stationary random fields. For example, there seems to be no consensus among statisticians as to the relative merits of various nonstationary models and their corresponding estimates found in the current literature (examples of such models can be found in \cite{hsing2016local,sampson1992nonparametric,paciorek2006spatial,fuglstad2015exploring}). While powerful spectral techniques have been developed for deriving absolute continuity or orthogonality of stationary random fields \cite{ibragimov2012gaussian} and for understanding the implications for spatial interpolation \cite{stein2012interpolation}, we know quite little about such topics for nonstationary random fields.  A further complicating matter is that even if the parametric form for the nonstationary data generating mechanism is known, difficulties associated with the inevitable increase in the number of unknown parameters and computational complexity can outweigh the benefits of fitting the true nonstationary model. The situation is far worse for spatial statistics---as compared to time series---where irregularity of spatial observation locations and large boundary effects can make estimation and modeling of nonstationarity more difficult.

In this paper, we attempt to shed light on a new class of nonstationary random fields which exhibit, what we call, {\em local invariant nonstationarity}. Formally defined in Section \ref{section: local invariant}, a locally invariant nonstationary random field $\{Z(\bs x)\colon \bs x\in \Bbb R^d\}$ has the feature that the covariance function can be written in the following form,
\begin{align*}
\text{cov}(Z(\bs x), Z(\bs y))=K(\bs x-\bs y, \bs \theta(\bs x)-\bs \theta(\bs y)),
\end{align*}
where $K(\cdot, \cdot)\colon \Bbb R^d\times \Bbb R^d \rightarrow \Bbb R$ and $\bs\theta(\cdot) \colon \Bbb R^d\rightarrow \Bbb R^d$ is a vector field which characterizes the nonstationarity in $Z(\bs x)$. We argue that the locally invariant property has a special interaction with a new  generalized quadratic estimate (derived in Section \ref{Section: the quad est}) which extends an estimate originally developed for gravitational lensing of the Cosmic Microwave Background in Cosmology \cite{hu2001mapping, hu2002mass}. The nature of this interaction not only encourages low estimation bias but also enables accurate (and fast) Frequentist mean square error quantification of the nonstationarity characterized by $\bs \theta(\bs x)$. These quadratic estimates are interesting in their own right, as they detect and estimate nonstationarity by probing correlation among Fourier frequencies, the absence of which is the characterizing feature of weak stationarity (by Bochner's Theorem). Moreover, this generalized quadratic estimate can be computed with a Fourier characterization that runs in $\mathcal O(n\log n)$ time and $\mathcal O(n)$ storage when observing the field on a uniform grid of size $n$ in $\Bbb R^d $. Finally, the work presented here partially addresses two unanswered questions that arise in two important bodies of work on nonstationarity: 1) estimating the spatially varying phase in the models analyzed by Dahlhaus \cite{dahlhaus1997fitting,dahlhaus2000likelihood} and 2) how to extend the quadratic estimate of dark matter from the Cosmic Microwave Background to more general nonstationary random fields.

The seminal work of Dahlhaus in the time series literature \cite{dahlhaus1997fitting,dahlhaus2000likelihood}  is  an example of a collection of results that hint at a more general statistical theory of nonstationarity.  Dahlhaus develops asymptotic theory for a particular class of nonstationary random fields modeled by a spatially (or temporally) varying spectral density. These random fields were originally developed for time series \cite{priestley1965evolutionary,priestley1981spectral} and have the form
\begin{align}
\label{eq: intro nonstat phase model}
Z(\bs x)= \int_{\Bbb R^d} e^{i\bs x\cdot \bs k} A(\bs k, \bs x)\sqrt{C_{\bs k}}\frac{dW_{\bs k}}{(2\pi)^{d/2}}
\end{align}
where  $C_{\bs k}$ is a spectral density, $dW_{\bs k}$ is an orthogonal increment random complex measure that satisfies $E|dW_{\bs k}|^2 =  d{\bs k}$ (see \cite{Gikhman_Skorokhod_book_v1} for details on random measures) and $A({\bs k},{\bs x})$ represents a spatial (or temporal) modulation of $\sqrt{C_{\bs k}}$.
Dahlhaus proved consistency results for estimating the squared modulus $|A({\bs k},{\bs x})|^2$ when estimation is done  by maximizing a weighted sum of local likelihood functions. Left unanswered, however, is the question of estimating the phase of $A({\bs k},{\bs x})$. In Section \ref{section: NPhase}, we partly resolve this question by showing that a generalized quadratic estimate can be used to estimate a pure phase modulation $A({\bs k},{\bs x})=\exp({iB(\bs k,\bs x)})$ where $B(\bs k,\bs x)$ is a function taking values in $\Bbb R$ and is a separable function of $\bs x$ and $\bs k$. Indeed, this nonstationary phase model has the local invariance property and, as such, can be accurately estimated (under certain conditions on $B(\bs k,\bs x)$) with the generalized quadratic estimate developed here.

Another important development in the statistical theory of nonstationarity comes from recent gravitational lensing studies of the Cosmic Microwave Background (CMB) \citep{das2011detection,van2012measurement, planck2013lensing, Polarbear2014, planck2015lensing}. In Cosmology, gravitational lensing describes the distortion of photon trajectories due to density fluctuations of intervening dark matter. These density fluctuations affect the CMB observations by introducing small nonstationarities in the original isotropic random field model of the CMB. The state-of-the-art estimator of lensing, the quadratic estimator developed by Hu and Okamoto \cite{hu2001mapping, hu2002mass}, has become an incredibly successful tool for probing the nature of dark matter, understanding cosmic structure and constraining cosmological parameters. What is so surprising about this estimate is that is has small bias. This is due to a delicate cancellation of terms in an infinite Taylor expansion of the lensing effect. Unfortunately there has been no clear explanation as to why this cancellation occurs and whether or not it exists in other models of nonstationarity. In this paper, we argue that this cancellation is due to the fact that the lensing-induced nonstationarity is locally invariant. Moreover, the generalized quadratic estimator developed here extends the lensing estimator to more general settings. Indeed, many of techniques we use to derive the generalized quadratic estimator are inspired by---and closely follow---those used by Hu and Okamoto \cite{hu2001mapping, hu2002mass}.  The point of this paper, in contrast to the work of Hu and Okamoto, is to identify the cause of the delicate Taylor series cancellation and extend the benefits of the lensing estimator to a larger class of nonstationary random fields available to general practitioners of spatial statistics.

The first part of this paper, given in Section~\ref{section: local invariant}, defines the {\it locally invariant} property and uses it to derive a corresponding generalized quadratic estimate, called the quadratic estimate hereafter,  which has particularly low bias. These new estimates are global rather than local in nature and thus avoid complicating theoretical and practical issues related to bandwidth selection. Moreover, they are unique in that they work in the spectral domain by estimating cross correlation of the Fourier coefficients.  In Section~\ref{section: local invariant} we also derive analytic approximations for estimation variance and second order bias of the quadratic estimate. These approximations, and indeed the estimate itself, are often very accurate and have Fourier representations that yield fast computation.

The second part of this paper is given in Section~\ref{section: NPhase}. Here we analyze random fields which are characterized by a spatially varying spectral phase modulation of a stationary field, called nonstationary spectral phase random fields for the remainder of this paper. These models effectively generalize warping models, are locally invariant and, as such, are amenable to quadratic estimates. In Subsection \ref{section: Locally attainable spectral densities} we characterize how the local spectral density of the nonstationary spectral phase model varies as a function of spatial location. The theory of optimal transport and the $L_2$-Wasserstein metric play an important role in this characterization and leads to a natural heuristic for quantifying estimation bias in terms of the Wasserstein geodesic cut locus (see Claim \ref{claim: cut locus}, Section \ref{section: modeling xi, C and eta} and Figure \ref{Figure 2 zx and var, extra bias}). We also present a set of simulations, based on nonstationary spectral phase random fields, which demonstrate the accuracy of both the quadratic estimate and our approximation to the mean squared sampling properties of the quadratic estimate.

\subsection{Notation}

For readability it will be advantageous to briefly summarize the notational conventions used throughout this paper, some of which are borrowed from Cosmology and are somewhat nonstandard in the statistics literature.
Variables taking values in $\Bbb R^d$ or $\Bbb R^{d\times d}$ will generally be written with bold font such as $\bs x, \bs y, \bs k,\bs \ell\in \Bbb R^d$ or $\bs A,\bs M \in \Bbb R^{d\times d}$. Indexing into vector or matrix coordinates are written with subscripts so that $\bs x_i\in\Bbb R$ denotes the $i^{\text{th}}$ coordinate of $\bs x\in\Bbb R^d$, for example. When $\bs z, \bs w\in\Bbb C^d$ we let the (non-Hermitian) dot product  be denoted by $\bs z \cdot \bs w = \bs z_1 \bs w_1 + \ldots + \bs z_d \bs w_d$.

Vector fields  $\bs \theta(\bs x):\Bbb R^d\rightarrow \Bbb R^d$ are also written bold so that \mbox{$\bs \theta(\bs x)=(\bs\theta_{1}(\bs x),\ldots, \bs\theta_{d}(\bs x))^T$} where $\bs \theta_i(\bs x):\Bbb R^d\rightarrow \Bbb R$. The Fourier transform of $\bs \theta(\bs x)$, for example, is applied coordinate-wise and written $\bs \theta_{\bs k}=(\bs \theta_{1,\bs k},\ldots, \bs \theta_{d,\bs k})^T$ where $\bs \theta_{j,\bs k}$ denotes the Fourier transform of $\bs \theta_j(\bs x)$ and is defined as
\[
 \bs\theta_{j,\bs k}= \int e^{-i\bs x\cdot \bs k} \bs \theta_{j}(\bs x) \frac{d\bs x}{(2\pi)^{d/2}}.
 \]
For a mean-zero stationary random field $\{Z(\bs x): \bs x\in\Bbb R^d\}$ the autocorrelation function is denoted  $C^Z(\bs x-\bs y) := E(Z(\bs x)Z(\bs y)^*)$ and the spectral density is denoted $C_{\bs\ell}^{ZZ} = {(2\pi)}^{d/2} C_{\bs\ell}^Z$ where, via our convention, $C_{\bs\ell}^Z$ denotes the Fourier transform of $C^Z(\bs x)$.
Parenthetical superscripts are reserved for enumerating functions (and \textit{not} higher order derivatives). For example $f^{(1)}_{\bs k}, f^{(2)}_{\bs k}, f^{(3)}_{\bs k}, \ldots$ denotes a sequence of functions taking arguments $\bs k\in \Bbb R^d$ in the Fourier domain. This convention avoids indexing ambiguities and the subscript convention of the Fourier transform.

In the derivations that follow, one may assume all random fields have periodic boundary conditions on $(-L/2, L/2]{}^d$, for some large $L>0$. This alleviates subtleties associated with the Fourier transform of non-periodic random fields defined on $\Bbb R^d$. However, extensions to non-periodic random fields can be made with an appropriate use of generalized random fields and generalized Fourier transforms. To incorporate the periodic case and the---possibly generalized---general case we use a single notation for the Fourier transform in both situations. For example when  $f(\bs x)$ is periodic on $(-L/2, L/2]{}^d$ the notation $\int e^{-i\bs x\cdot \bs k} f(\bs x)\frac{d\bs x}{(2\pi)^{d/2}}$ and $\int e^{i\bs x\cdot \bs k} f_{\bs k}\frac{d\bs k}{(2\pi)^{d/2}}$ should be interpreted as notationally equivalent to $\int_{-L/2}^{L/2}\cdots\int_{-L/2}^{L/2} e^{-i\bs x\cdot \bs k} f(\bs x)\frac{d\bs x}{(2\pi)^{d/2}}$ and  $\sum_{\bs k \in \frac{2\pi}{L}\Bbb Z^d }   e^{i \bs x\cdot \bs k}  f_{\bs k} \frac{(2\pi/ L){}^{d}}{(2\pi){}^{d/2}} $, respectively. A consequence of this convention is that, in the periodic case, one equates $d\bs k$ with $(2\pi / L)^d$ and, therefore,  the Dirac delta function $\delta_{\bs k}$ becomes a regular function taking the value $1/d\bs k$ when $\bs k=\bs 0$ and zero otherwise. In particular, if $Z(\bs x)$ is a mean zero stationary random field with with periodic boundary conditions on $(-L/2, L/2]^d$ then $E(Z_{\bs k}^{\phantom{*}} Z^{*}_{\bs \omega}) = \delta_{\bs k-\bs \omega}C_{\bs k}^{ZZ}$
and $E(|Z_{\bs k}|^2) = \delta_{\bs 0}C_{\bs k}^{ZZ}$.

\section{Locally invariant nonstationary random fields}
\label{section: local invariant}

In this section we define a property of nonstationary random fields called \textit{local invariance}. This property,  along with a small set of generic model and observational assumptions, appears to be an important ingredient for quadratic estimates of nonstationarity to have low bias. Indeed, this is the main theme of the paper: {\it that the structure of local invariance encourages bias cancellation}. A secondary theme of this paper is that local invariance provides a flexible restriction to the class of all random field covariance functions whereby making generalized quadratic estimation available to a wide class of nonstationary random field applications.

\begin{definition}
\label{Definition 1}
Let $C^{\bs \theta}(\bs x, \bs y)$ be a positive definite covariance function defined on $\bs x, \bs y\in\Bbb R^d$ and parameterized by a vector field $\bs \theta(\bs x)\colon\Bbb R^d\rightarrow \Bbb R^d$.
Then $\bs \theta(\bs x)$ is said to be a \textbf{local invariant for $C^{\bs \theta}(\bs x,\bs y)$} if there exists a function \mbox{$K:\Bbb R^d\times \Bbb R^d \rightarrow \Bbb R$} such that
\begin{align}\label{def eq: local invariant covariance function}
C^{\bs \theta}(\bs x,\bs y) = K\big(\bs x-\bs y,\bs \theta(\bs x) -  \bs \theta(\bs y)\big).
\end{align}
Equivalently, \textbf{$Z(\bs x)$ is  locally invariant with respect to $\bs\theta(\bs x)$} if $Z(\bs x)$ is a random field  with covariance function $C^{\bs \theta}$ that satisfies (\ref{def eq: local invariant covariance function}).
\end{definition}

The name \textit{local invariance} is intended to express the following fact: any region where $\bs \theta(\bs x)$ is constant results in the same local stationary model. In particular, suppose $Z(\bs x)$ is a nonstationary random field with covariance function $C^{\bs \theta}(\bs x, \bs y)$ satisfying (\ref{def eq: local invariant covariance function}). If $\bs \theta(\bs x)$ has no variation (i.e. is constant) over neighborhoods $\Omega_1\subset \Bbb R^d$ and $\Omega_2\subset \Bbb R^d$ then $Z(\bs x)$ is locally stationary over $\Omega_1$ and $\Omega_2$ with the same local autocovariance function $K(\bs x-\bs y, \bs 0)$.

Examples of local invariant nonstationary random fields are easy to find. Indeed any warped random field of the form $Z(\bs x + \bs \theta(\bs x))$ is locally invariant with respect to $\bs\theta(\bs x)$ when $Z$ is a stationary random field. Another example of a locally invariant model, discussed later in Section \ref{section: NPhase}, is the spatially varying spectral phase model given in (\ref{eq: intro nonstat phase model}) where $A(\bs k,\bs x)=\exp(i\bs\theta(\bs x) \cdot \bs\eta_{\bs k})$ and $\bs \eta_{\bs k}$ is a known function mapping $\Bbb R^d$ into $\Bbb R^d$ that has odd symmetry about the origin. It is interesting to note that many of the asymptotic results for spatially (or temporally) varying spectral models study the estimation of $|A(\bs k, \bs x)|^2$ using a local periodogram \cite{dahlhaus1997fitting} or a version or the preperiodogram \cite{dahlhaus2000likelihood}. Since  $|A(\bs k, \bs x)|^2 = 1$ for the spectral phase model (\ref{eq: intro nonstat phase model}), these local periodograms to not immediately apply to the estimation of $A(\bs k, \bs x)$ defined in (\ref{eq: intro nonstat phase model}).

In what follows we derive a quadratic estimate of $\bs\theta(\bs x)$ for locally invariant nonstationarity of the form given in Definition \ref{Definition 1}. The estimates are derived under the following observational scenario: a single realization of the nonstationary random field observed on a dense grid with additive stationary noise. These assumptions hold throughout the paper. We list them here to be completely explicit.

\begin{assumption}\label{Assumption 4}
Let $Z(\bs x)$ be a mean zero nonstationary Gaussian random field with local invariant nonstationary covariance function $C^{\bs \theta}(\bs x, \bs y)$ satisfying Definition \ref{Definition 1}. The data field, denoted $Z^{obs}(\bs x)$, is observed on a dense regular grid and has the form
\[
Z^{obs}(\bs x) = Z(\bs x) + N(\bs x)
\]
where  $N(\bs x)$ is a mean zero stationary Gaussian generalized random noise with spectral density $C_{\bs k}^{N\!N}$.
\end{assumption}

Our second assumption is that the local invariant vector field $\bs \theta(\bs x)$ can be additionally characterized by an unknown scalar potential function $\phi(\bs x):\Bbb R^d \rightarrow \Bbb R$. This assumption simply  reduces the amount of complexity necessary for  developing a quadratic estimate of $\bs \theta(\bs x)$ while still retaining enough modeling flexibility.

\begin{assumption}\label{Assumption 2}
Suppose the vector field $\bs\theta(\bs x)\colon \Bbb R^d \rightarrow \Bbb R^d$ is characterized by an unknown scalar potential field $\phi(\bs x)\colon \Bbb R^d \rightarrow \Bbb R$ along with a known vector field  $\bs \xi_{\bs k}$ which satisfies
\[
\bs\theta_{\bs k} := \big(\bs\xi^{\phantom{*}}_{1,\bs k} \hspace{.05cm}\phi_{\bs k}^{\phantom{*}}\hspace{.05cm}, \,\ldots, \, \bs \xi_{d,\bs k}^{\phantom{*}}\hspace{.05cm}\phi_{\bs k}^{\phantom{*}}\big)^T\!\!.
\]
For each $j\in\{1,\ldots, d\}$ the coordinate spectral multiplier $\bs\xi_{j,\bs k}:\Bbb R^d\rightarrow \Bbb C$ is assumed to be a Hermitian function of $\bs k$ so that $\bs\xi_{j,-\bs k}^{\phantom{*}}= \bs\xi_{j,\bs k}^*$.
\end{assumption}

Notice that the above scalar potential model includes the case that $\bs \theta(\bs x)=\nabla \phi(\bs x)$,  by setting $\bs \xi_{\bs k}=i\bs k$, and the case that $\bs \theta(\bs x) = (\phi(\bs x),\ldots, \phi(\bs x))$, by setting $\bs \xi_{\bs k}=(1,\ldots, 1)$.

\begin{assumption}\label{Assumption 3}
There exists a mean zero stationary Gaussian random field prior for the unknown scalar potential $\phi(\bs x)\colon \Bbb R^d \rightarrow \Bbb R$. Let $C^{\phi}(\bs x - \bs y)= E(\phi(\bs x)\phi(\bs y))$ denote the autocovariance function for $\phi(\bs x)$ and $C^{\phi\phi}_{\bs k}$ denote the corresponding spectral density for this prior.
\end{assumption}

It is important to note that the prior in Assumption \ref{Assumption 3} is not used to generate a Bayesian posterior. A Bayesian posterior sampling methodology would be an exciting development but not the scope of the current paper. Instead, the prior is only used to approximate the marginal distribution of the data which, in turn, is used to optimize Fourier weights and to generate a variance approximation for the quadratic estimate of $\phi_{\bs k}$. Indeed the quadratic estimate, derived in Section \ref{Section: the quad est}, is defined to be an unbiased estimate of $\phi$ (up to first order) regardless of how $C_{\bs k}^{\phi\phi}$ is specified. The only effect of mis-specification (or non-existence) of $C_{\bs k}^{\phi\phi}$ will be to generate an estimate which gives too much weight to unruly frequencies and to report a less accurate mean squared error when using the approximations developed in Section \ref{Section: var and bias}.

Assumptions \ref{Assumption 4}, \ref{Assumption 2} and \ref{Assumption 3} are generic and intended to isolate a small set of  assumptions for deriving an estimate with small bias. Bias is not universally guaranteed to be small but the generalized quadratic estimates, derived in the next section for  local invariance models, often have surprisingly small estimation bias due to the local invariant property.

\subsection{
A generalized quadratic estimate of $\phi(\bs x)$
}\label{Section: the quad est}

Based on assumptions \ref{Assumption 4}, \ref{Assumption 2} and \ref{Assumption 3} given in the previous section, the first step for deriving a generalized quadratic estimate of $\phi(\bs x)$ is to linearly approximate  $C^{\bs \theta}(\bs x,\bs y)$, with a Taylor expansion in $\bs \theta(\bs x) - \bs \theta(\bs y)$ as follows
\begin{align}
\label{first order term, intro}
C^{\bs \theta}(\bs x,\bs y) = C^{(0)}(\bs x-\bs y)
        + \bs C^{(1)}(\bs x-\bs y)\cdot \left({\bs\theta(\bs x)-\bs\theta(\bs y)}\right)
        + \mathcal O({\phi}^2)+ \mathcal O(\phi^3)+\cdots.
\end{align}
where $C^{(0)}\colon \Bbb R^d \rightarrow \Bbb R$ and $\bs C^{(1)}\colon \Bbb R^d \rightarrow \Bbb R^d$ satisfies $\bs C^{(1)}(-\bs x) = -\bs C^{(1)}(\bs x)$.
Now, truncating (\ref{first order term, intro}) to first order and applying Claim \ref{claim: first order expansion} of Appendix \ref{section: Detailed Proofs} gives the following linear approximation of the cross frequency covariance in the Fourier transform of $Z(\bs x)$
\begin{align}
\label{eq: first order cov intro}
E\big(Z_{\bs k+\bs \ell}Z_{-\bs k}\big)  \approx  \phi_{\bs\ell}\left({\bs\xi_{\bs\ell}\!\cdot\! \bs C^{(1)}_{\bs k} - \bs\xi_{\bs\ell}\!\cdot\! \bs C_{\bs k+\bs \ell}^{(1)}}\right).
\end{align}
Recall Bochner's Theorem (details can be found in \cite{Gikhman_Skorokhod_book_v1}) which states that the random field $Z(\bs x)$ is stationary if and only if  \mbox{$E\big(Z_{\bs k+\bs \ell}Z_{-\bs k}\big) = 0$} for every $\bs\ell \neq\bs 0$. Therefore nontrivial covariance between $Z_{\bs k+\bs \ell}$ and $Z_{-\bs k}$ at some nonzero lag $\bs\ell$ provides a direct probe into the nonstationarity present in $Z(\bs x)$. Equation (\ref{eq: first order cov intro}) is, therefore, a translation of how local invariant nonstationarity relates to nonzero cross covariance in $Z_{\bs k}$.
Using this translation, Claim \ref{appendix, claim: quad est derivations} of Appendix \ref{section: Detailed Proofs} derives the quadratic estimate of $\phi_{\bs \ell}$. This estimate is effectively  an inverse-variance weighted autocovariance estimate in the Fourier domain and is given by
\begin{align}
    \label{explicit form of the qe in body of paper}
    \hat\phi_{\bs \ell}
    &=  A_{\bs \ell}
        \int
        {\Big(\bs\xi_{\bs \ell} \!\cdot\!\bs C^{{(1)}}_{\bs k} - \bs\xi_{\bs \ell} \!\cdot\!\bs C^{{(1)}}_{\bs k+\bs \ell}\Big)}^{\! *}
        \frac{Z^{obs}_{\bs k+\bs \ell}Z^{obs}_{-\bs k}}{C^{ZZobs}_{\bs k+\bs \ell}C^{ZZobs}_{\bs k}}\frac{d\bs k}{{(2\pi)}^{d/2}}.
\end{align}
In the above formula, $A_{\bs \ell}$ is a normalizing constant (see Claim \ref{appendix, claim: quad est derivations} for an exact expression),  $C^{ZZobs}_{\bs k}$ is defined to be $C^{ZZm}_{\bs k} + C^{N\!N}_{\bs k}$ where $C^{ZZm}_{\bs k}$ is the spectral density of $Z(\bs x)$ marginalized over the prior for $\phi(\bs x)$ given in Assumption \ref{Assumption 3}. Notice that the prior only serves to optimize the weights in $\hat\phi_{\bs \ell}$. Indeed, one can easily avoid specifying $C_{\bs k}^{\phi\phi}$ by instead defining $C^{ZZm}_{\bs k}$ to be ${(2\pi)}^{d/2}C^{(0)}_{\bs k}$ where $C^{(0)}_{\bs k}$ denotes the Fourier transform of $C^{(0)}(\bs x)$ from (\ref{first order term, intro}).

The normalizing constant $A_{\bs \ell}$ is completely determined by the weights used on the terms $Z^{obs}_{\bs k+\bs \ell}Z^{obs}_{-\bs k}$ through the requirement that $\hat\phi_{\bs \ell}$ be unbiased up to first order in $\phi$, in particular, requiring that $E(\hat\phi_{\bs \ell}) = \phi_{\bs \ell} + \mathcal O(\phi^2)+ \mathcal O(\phi^3)+\cdots$ or equivalently that $\frac{1}{\phi_{\bs \ell}}E(\hat\phi_{\bs \ell}) = 1 + \mathcal O(\phi)+ \mathcal O(\phi^2)+\cdots$.
This results in the following an analytic characterization for $A_{\bs \ell}$
\begin{equation}
\label{in text eq of the normalizing constant}
 A_{\bs \ell}
        \int
        \frac{\big|{\bs\xi_{\bs\ell}\!\cdot\! \bs C^{(1)}_{\bs k} - \bs\xi_{\bs\ell}\!\cdot\! \bs C_{\bs k+\bs \ell}^{(1)}}\big|^2}{C^{ZZobs}_{\bs k+\bs \ell}C^{ZZobs}_{\bs k}}\frac{d\bs k}{{(2\pi)}^{d/2}} = 1 + \mathcal O(\phi) + \mathcal O(\phi^2) + \cdots.
\end{equation}
Notice also that one is free to manually change weights used on each term $Z^{obs}_{\bs k+\bs \ell}Z^{obs}_{-\bs k}$ in  (\ref{explicit form of the qe in body of paper}). This may be advantageous for optimizing the sampling behavior of $\hat\phi_{\bs \ell}$ to specific applications. For example one may want to down-weight corrupted frequencies in a particular experimental setting. In this case, however, the form of the normalizing constant $A_{\bs \ell}$ given in (\ref{in text eq of the normalizing constant}) will need to be adjusted accordingly.

One of the advantages of the estimator $\hat\phi_{\bs \ell}$, defined in (\ref{explicit form of the qe in body of paper}), is that there exists a fast algorithm for computing $\hat\phi_{\bs \ell}$ at all frequencies $\bs \ell$ simultaneously by alternating pointwise operations in the Fourier domain and the pixel domain. Indeed by Claim \ref{appendix, claim: quad est derivations} of Appendix \ref{section: Detailed Proofs} the quadratic estimate given in (\ref{explicit form of the qe in body of paper}) is equivalent to
\begin{align}
    \label{fast form of the qe in body of paper}
    \hat\phi_{\bs \ell}
    &=   A_{\bs \ell}\sum_{p=1}^d \bs\xi^*_{p,\bs \ell} \int
        e^{-i\bs x\cdot\bs \ell} \mathscr A(\bs x)\mathscr B_{p}(\bs x)\frac{d\bs x}{{(2\pi)}^{d/2}}
\end{align}
where $\mathscr A_{\bs \ell}:= Z^{obs}_{\bs \ell}/ C^{ZZobs}_{\bs \ell}$ and $\mathscr B_{p,\bs \ell} := i 2  \, \textrm{imag}(\bs C^{(1)}_{p,\bs \ell}) Z^{obs}_{\bs\ell} / C^{ZZobs}_{\bs\ell}$ which  can be computed in $\mathcal O(n\log(n))$ time (when observing $Z^{obs}$ on a grid of size $n$) by a sequence of fast Fourier transforms, inverse fast Fourier transforms and pointwise operations.

The derivation above only depends on the local invariance property insofar as it is used to optimize the weights  in (\ref{explicit form of the qe in body of paper}) and the resulting normalizing constant given in (\ref{fast form of the qe in body of paper}). Indeed, exactly similar arguments can be made for deriving quadratic estimates of nonstationary models which are not locally invariant, such as covariance functions of the form $C^{\bs\theta}(\bs x,\bs y)=K(\bs x-\bs y,\bs\theta(\bs x)+\bs\theta(\bs y))$ for example. The key difference is that the quadratic estimate $\hat \phi_{\bs \ell}$ for models which are {\it not} locally invariant  tend to either have a large  $\mathcal O(\phi^2)+ \mathcal O(\phi^3)+\cdots$ bias,  small signal to noise ratio, or have significant non-Gaussian estimation variability. Local invariant models, in contrast, encourage a significant amount of cancellation occurring within $\mathcal O(\phi^2)+ \mathcal O(\phi^3)+\cdots$ so that bias is small even in the regime of moderately large signal to noise ratio. Moreover, small higher order terms provide a regime where mean square estimation variability is accurately approximated with easily computable formulas. This is explored in more detail in Section \ref{Section: var and bias} and in the simulation examples presented later.

\subsection{The Hu and Okamoto lensing estimate as a special case of $\hat\phi_{\bs \ell}$}
\label{section: Hu and Okamoto lensing special case}

In this section we show that $\hat \phi_{\bs \ell}$, derived in the previous section, is an extension of the original quadratic estimate developed in \cite{hu2001mapping, hu2002mass} for Cosmic Microwave Background gravitational lensing. Start by letting $Z(\bs x) = T(\bs x + \nabla \phi(\bs x))$  denote the lensed Cosmic Microwave Background  and $\phi(\bs x)$ denote the projected gravitational potential in the $\Bbb R^2$ flat sky approximation. In the original derivation \cite{hu2001mapping, hu2002mass} a Taylor approximation is first applied to $T(\bs x+ \nabla \phi(\bs x))$ as follows
\begin{equation}
\label{HuOk approx 1}
Z(\bs x) \approx T(\bs x) +  \nabla T(\bs x)\cdot\nabla \phi(\bs x)
\end{equation}
The above linear model is then used to linearly approximate $Z(\bs x)Z(\bs y)$ by additionally discarding the term $\big(\nabla T(\bs x)\cdot\nabla \phi(\bs x)\big) \big(\nabla T(\bs y)\cdot\nabla \phi(\bs y)\big)$ which is quadratic in $\phi$. Taking Fourier transforms, then an expected value, results in the following approximation
\begin{equation}
\label{HuOk approx 2}
E(Z_{\bs k+\bs \ell}Z_{-\bs k}) \approx \frac{\phi_{\bs \ell}}{2\pi} \Big(\bs \ell\cdot(\bs k+\bs \ell)C^{TT}_{\bs k+\bs \ell} - \bs \ell\cdot\bs k C^{TT}_{\bs k}\Big).
\end{equation}
Notice that (\ref{HuOk approx 2}) is a special case of (\ref{eq: first order cov intro}), and hence a special case of $\hat\phi_{\bs \ell}$ in (\ref{explicit form of the qe in body of paper}), when setting $\bs \xi_{\bs \ell} = i\bs \ell$ and $\bs C^{(1)}_{\bs k} = \frac{i\bs k}{2\pi}  C^{TT}_{\bs k}$.

It is important to notice a particular subtlety when analyzing the accuracy of  (\ref{HuOk approx 2}) in terms of the magnitude of the discarded terms in the Taylor approximation (\ref{HuOk approx 1}). This subtlety can be illustrated by assuming the displacement $\nabla \phi(\bs x)$ is extremely large and happens to be a \textit{constant function of $\bs x$}. In this case one clearly has $Z(\bs x) \not\approx T(\bs x)$, i.e. the zeroth order Taylor approximation completely breaks down. Yet, in a distributional sense, the zeroth order Taylor approximation is perfect since $Z(\bs x)$ and $T(\bs x)$ have the same finite dimensional distributions (by the fact that $T(\bs x)$ is isotropic and $\nabla \phi(\bs x)$ is assumed to be constant).
Therefore one can not quantify the accuracy of (\ref{HuOk approx 2}) by a map level analysis of the individual discarded terms in (\ref{HuOk approx 1}).
In fact we propose that (\ref{HuOk approx 2}) is accurate, not because the map level Taylor approximation (\ref{HuOk approx 1}) is good (for which it is not), but rather because $\nabla \phi(\bs x)$ is a locally invariant parameter and the corresponding nonstationary covariance
\[
C^{\bs \theta}(\bs x,\bs y)=E(Z(\bs x)Z(\bs y)) = C^T(\bs x - \bs y + \nabla\phi(\bs x) - \nabla\phi(\bs y))
\]
has an accurate Taylor approximation in $\nabla\phi(\bs x) - \nabla\phi(\bs y)$, vis-\`a-vis (\ref{first order term, intro}).

\subsection{Variance and bias analytic approximation}
\label{Section: var and bias}

In Section \ref{Section: the quad est} a fast formula was derived for computing the estimate $\hat \phi_{\bs \ell}$ when observing a single realization of $Z^{obs}(\bs x)$. The speed at which $\hat \phi_{\bs \ell}$ can be computed on a dense observation grid makes it possible to perform large scale Monte Carlo analysis on $\hat\phi_{\bs \ell}$ in any experimental setting for which $Z^{obs}(\bs x)$ can be easily simulated. In this section we complement a simulation-based method of uncertainty quantification by providing analytic approximations to variance and second order bias of $\hat\phi_{\bs \ell}$. These approximations are often very accurate and inherit a similar Fourier representation as (\ref{fast form of the qe in body of paper}) for fast computation.

By inspection of (\ref{explicit form of the qe in body of paper}) one can consider $\hat\phi_{\bs \ell}$ as a function of the quadratic form  $Z^{obs}_{\bs k+\bs \ell}Z^{obs}_{-\bs k}$, integrating over the variable $\bs k$. In what follows we will consider the sampling behavior of $\hat\phi_{\bs \ell}$ when replacing $Z^{obs}_{\bs k+\bs \ell}Z^{obs}_{-\bs k}$ by some other function $X_{\bs k, \bs \ell}$ of two variables $\bs k,\bs\ell\in \Bbb R^d$. The following definition sets notation for this operation which is useful for denoting terms which are related to variance and bias of the estimator $\hat\phi_{\bs \ell}$ derived in subsections \ref{SubSection: var} and \ref{SubSection: bias}.

\begin{definition}
\label{def: quad est applied to X}
For any function $X_{\bs k, \bs \ell}:\Bbb R^d\times \Bbb R^d \rightarrow \Bbb C$ let
$\hat\phi_{\bs \ell}\{X_{\bs k,\bs \ell}\}$ denote the quadratic estimate $\hat\phi_{\bs \ell}$ defined in (\ref{explicit form of the qe in body of paper}) but applied to the function $X_{\bs k, \bs \ell}$ rather than $Z^{obs}_{\bs k+\bs \ell}Z^{obs}_{-\bs k}$. In particular $\hat\phi_{\bs \ell}\{X_{\bs k,\bs \ell}\}$ is a function of $\bs \ell$  and satisfies
\[
\hat\phi_{\bs \ell}\{X_{\bs k,\bs \ell}\}
:=A_{\bs \ell}
        \int
        {\Big(\bs\xi_{\bs \ell} \!\cdot\!\bs C^{{(1)}}_{\bs k} - \bs\xi_{\bs \ell} \!\cdot\!\bs C^{{(1)}}_{\bs k+\bs \ell}\Big)}^{\! *}
        \frac{X_{\bs k,\bs \ell}}{C^{ZZobs}_{\bs k+\bs \ell}C^{ZZobs}_{\bs k}}\frac{d\bs k}{{(2\pi)}^{d/2}}.
\]
If, on the other hand, $X_{\bs k}$ and $Y_{\bs k}$ are both functions of a single frequency argument $\bs k\in \Bbb R^d$ then  we define
\[
\hat\phi_{\bs \ell}\{X,\!Y\}
:=A_{\bs \ell}
        \int
        {\Big(\bs\xi_{\bs \ell} \!\cdot\!\bs C^{{(1)}}_{\bs k} - \bs\xi_{\bs \ell} \!\cdot\!\bs C^{{(1)}}_{\bs k+\bs \ell}\Big)}^{\! *}
        \frac{X_{\bs k+\bs \ell}Y_{-\bs k}}{C^{ZZobs}_{\bs k+\bs \ell}C^{ZZobs}_{\bs k}}\frac{d\bs k}{{(2\pi)}^{d/2}}.
\]
\end{definition}

In the following two sections we derive approximations to the mean squared error and bias when using $\hat\phi_{\bs \ell}$ to estimate $\phi_{\bs \ell}$. This comes in the form of two functions $C_{\bs \ell}^{\text{var }\hat\phi}$ and $C_{\bs \ell}^{\text{bias }\hat\phi}$ which represent approximations to the spectral density of variance and bias after marginalizing over the unknown $\phi_{\bs \ell}$ using the Gaussian random field prior given in Assumption \ref{Assumption 3}.

\subsubsection{Variance spectral density $C_{\bs \ell}^{\text{var }\hat\phi}$}
\label{SubSection: var}

There are two main contributions to the variability in $\hat\phi_{\bs\ell}$. The first source of variability is due to the   additive observational noise $N(\bs x)$. The second source, sometimes called \textit{shape noise} in Cosmology, is due to the baseline stationary fluctuations in $Z(\bs x)$ characterized by the autocovariance function $K(\bs x - \bs y, \bs 0)$.  The spectral density of this shape noise can be approximated by ${(2\pi)}^{d/2}C^{(0)}_{\bs \ell}$, which corresponds to the zero${}^\text{th}$ order approximation in (\ref{first order term, intro}), or by $C^{ZZm}_{\bs \ell}$ which denotes the  spectral density of $Z(\bs x)$ marginalized over the prior from Assumption \ref{Assumption 3}. Both these approximations can be used, within the formulas derived below, to give accurate approximations to the mean squared variability in $\hat\phi_{\bs \ell}$. However, in the cosmology literature on gravitational lensing, the marginal stationary model for $Z(\bs x)$ is typically used, rather than ${(2\pi)}^{d/2}C^{(0)}_{\bs \ell}$, for shape noise quantification.

To derive $C_{\bs \ell}^{\text{var }\hat\phi}$ first let $X(\bs x)$ denote the mean zero stationary Gaussian random field which models the sum of the observational noise $N(\bs x)$ and the shape noise discussed in the previous paragraph. By propagating the random field $X$ through the quadratic estimate one obtains an estimate of variability of $\hat\phi_{\bs \ell}$ around its expected value. In particular
\begin{align}
E\big(\big[\hat\phi_{\bs \ell}-E(\hat\phi_{\bs \ell})\big]\big[\hat\phi_{\bs \ell^\prime}-E(\hat\phi_{\bs \ell^\prime})\big]^*\big) &\approx E\big(\hat\phi_{\bs \ell}\{X,\!X\} \, \hat\phi_{\bs \ell^\prime}\{X,\!X\}^*\big) =:  \delta^{\phantom{*}}_{\bs \ell-\bs \ell^\prime} C^{\text{var }\hat\phi}_{\bs \ell} \label{var aprox wick form}
\end{align}
where the existence of the spectral density $C^{\text{var }\hat\phi}_{\bs \ell}$ is guaranteed by the fact that $X(\bs x)$ is stationary so that $\hat\phi_{\bs \ell}\{X,X\}$ is stationary in the pixel domain (see Claim \ref{thm: quantify the var fluctuations of hat phi}).
Depending on which approximation one uses for the baseline stationary fluctuations in $Z(\bs x)$, the spectral density of $X(\bs x)$ can be defined in one of the following two ways
\begin{align}\label{Cxx def}
C^{X\!X}_{\bs \ell}=
\begin{cases}
C^{N\!N}_{\bs \ell} + {(2\pi)}^{d/2}C^{(0)}_{\bs \ell}, & \text{option 1;}\\
C^{N\!N}_{\bs \ell} + C^{ZZm}_{\bs \ell}, & \text{option 2.}
\end{cases}
\end{align}
Now given $C^{X\!X}_{\bs \ell}$, Claim \ref{thm: quantify the var fluctuations of hat phi} establishes that
\begin{align}
	\label{eq: spec den of hatphi on X}
	C^{\text{var }\hat\phi}_{\bs \ell} &=
        2 A_{\bs\ell}^2
        \int
        {\Big|\bs\xi_{\bs \ell} \!\cdot\!\bs C^{{(1)}}_{\bs k} - \bs\xi_{\bs \ell} \!\cdot\!\bs C^{{(1)}}_{\bs k+\bs \ell}\Big|}^{2}
        \frac{C^{X\!X}_{\bs k+\bs \ell}}{{(C^{ZZobs}_{\bs k+\bs \ell})}{}^2}
        \frac{C^{X\!X}_{\bs k}}{{(C^{ZZobs}_{\bs k})}{}^2}\frac{d\bs k}{{(2\pi)}^{d}}.
    \end{align}
 In certain situations the right hand side of (\ref{eq: spec den of hatphi on X}) can be simplified. Recall that in the definition of $\hat\phi_{\bs \ell}$, one has two options for defining $C^{ZZobs}_{\bs \ell}$, either $C^{N\!N}_{\bs \ell} + {(2\pi)}^{d/2}C^{(0)}_{\bs \ell}$ or $C^{N\!N}_{\bs \ell} + C^{ZZm}_{\bs \ell}$, depending on if one wants to use the prior spectral density $C^{\phi\phi}_{\bs \ell}$ for optimizing the quadratic estimate weights. If the choice of $C^{ZZobs}_{\bs \ell}$ matches the choice of  $C_{\bs \ell}^{X\!X}$ then one obtains the following simplification of (\ref{eq: spec den of hatphi on X})
\begin{align}
	\label{eq: spec den of hatphi on CZZmobs}
C^{\text{var }\hat\phi}_{\bs \ell} = 2{(2\pi)}^{-d/2} A_{\bs \ell}.
 \end{align}

It should be noted that when marginalizing  over the prior given in Assumption \ref{Assumption 3} the process $Z(\bs x)$ becomes stationary but non-Gaussian. On the other hand, when conditioning on $\phi$, the process $Z(\bs x)$ is Gaussian but nonstationary. Therefore, when using option 2 in equation (\ref{Cxx def}) to model $C^{\text{var }\hat\phi}_{\bs \ell}$, the approximation in (\ref{var aprox wick form}) includes a Gaussian approximation to $X(\bs x)$. Finally, we mention that Claim \ref{thm: quantify the var fluctuations of hat phi} also gives a Fourier based characterization  for fast computation of $C^{\text{var }\hat\phi}_{\bs \ell}$.

\subsubsection{Bias spectral density $C_{\bs \ell}^{\text{bias }\hat\phi}$}
\label{SubSection: bias}

The higher order terms $\mathcal O(\phi^n)$ in  (\ref{first order term, intro}) are the exclusive source of bias in the quadratic estimate. The relation between $\mathcal O(\phi^n)$ and estimation bias can be written as follows
\[
E(\hat\phi_{\bs \ell}|\phi) -\phi_{\bs \ell} =   \hat\phi_{\bs\ell}\bigl \{ {\mathcal O(\phi^2)}_{\bs k +\bs \ell, -\bs k} \bigr \} + \hat\phi_{\bs\ell}\bigl \{ {\mathcal O(\phi^3)}_{\bs k +\bs \ell, -\bs k} \bigr \} +\cdots
\]
where $\mathcal O(\phi^n)_{\bs k+\bs \ell, -\bs k}$ is defined to be Fourier transform of $\mathcal O(\phi^n)(\bs x,\bs y)$, defined in (\ref{first order term, intro}), and evaluated at frequencies $\bs k+\bs \ell$ and  $-\bs k$, respectively. For the local invariant models we consider here, the second order bias term has the following form
\begin{equation}\label{eq: 2nd order bias (x,y)}
\mathcal O(\phi^2)(\bs x,\bs y) = (\bs\theta(\bs x) - \bs\theta(\bs y))^T \bs C^{(2)}(\bs x-\bs y) (\bs\theta(\bs x) - \bs\theta(\bs y))
\end{equation}
where $\bs C^{(2)}\colon \Bbb R^d\rightarrow \Bbb R^{d\times d}$ is a symmetric function about the origin.
This expression makes it clear how local invariant nonstationarity encourages low quadratic estimation bias. When the observational noise level is small, the high frequency fluctuations in the observations $Z^{obs}(\bs x)$ are more influential to the quadratic estimate.
At these small scales the smoothness of $\bs \theta(\bs x)$ and the function $\bs C^{(2)}(\bs x-\bs y)$ will attenuate the influence of $\mathcal O(\phi^2)$  when propagated through $\hat\phi_{\bs \ell}$.

For remainder of this section we analyze how the second order term (\ref{eq: 2nd order bias (x,y)}) propagates to second order bias in the quadratic estimate, denoted $\hat\phi_{\bs \ell}^{\text{bias}}$.
Claim \ref{thm: bias in the appendix} in the Appendix gives the following expression for $\hat\phi_{\bs \ell}^{\text{bias}}$
\begin{align}
    \label{2nd order bias in local invariant section II}
    \hat\phi_{\bs \ell}^{\text{bias}}&=\hat\phi_{\bs \ell}\big\{\mathcal O(\phi^2)_{\bs k+\bs \ell, -\bs k}\big\} =  2\sum_{p,q=1}^d\int  \bs\theta_{p,\bs\omega}\bs\theta_{q,\bs \ell - \bs\omega}\, \hat\phi_{\bs \ell}\big\{{\bs C^{(2)}_{p,q,\bs k}}- {\bs C^{(2)}_{p,q,\bs k + \bs\omega}}\big\}\frac{d\bs\omega}{(2\pi)^{d/2}}.
\end{align}
Moreover the marginal expected value of this bias term satisfies $E\big( \hat\phi_{\bs \ell}^{\text{bias}} \big) = 0$
when $\bs \ell \neq 0$. Therefore, to quantify the marginal magnitude of the second order bias one must use the variance of (\ref{2nd order bias in local invariant section II}). This is done in  Claim \ref{thm: bias in the appendix} which establishes that when $\bs\theta(\bs x)$ is a mean zero Gaussian random field with spectral density matrix $C_{\bs \ell}^{\bs \theta\bs\theta}$ the corresponding spectral density for $\hat\phi_{\bs \ell}^{\text{bias}}$, denoted $C_{\bs \ell}^{\text{bias }\hat\phi}$, satisfies
\begin{align}
	C_{\bs \ell}^{\text{bias }\hat\phi}
	&= 4\!\! \sum_{p,q,p^\prime\!,q^\prime\!=1}^d \int
	\Big(
	C^{\bs \theta\bs \theta}_{p,p^\prime,\bs \omega}C^{\bs \theta\bs \theta}_{q,q^\prime,\bs \ell - \bs \omega}
	+ C^{\bs \theta\bs \theta}_{p,q^\prime,\bs \omega}C^{\bs \theta\bs \theta}_{q,p^\prime,\bs \ell - \bs \omega}
	\Big)\nonumber
	\\
	&\qquad\qquad\qquad\qquad\qquad\times
	\hat\phi_{\bs \ell}\big\{{\bs C^{(2)}_{p,q,\bs k}}- {\bs C^{(2)}_{p,q,\bs k + \bs\omega}}\big\}
	\hat\phi_{\bs \ell}\big\{{\bs C^{(2)}_{p^\prime,q^\prime,\bs k}}- {\bs C^{(2)}_{p^\prime,q^\prime,\bs k + \bs\omega}}\big\}^*\frac{d\bs\omega}{(2\pi)^{d}}
	\label{eq in intro on spec den of O(phi2)}
\end{align}
when $\bs \ell \neq \bs 0$.
Notice that $\hat\phi_{\bs \ell}^{\text{bias}}$ equals the exact, map level, bias contribution from the second order term $\mathcal O(\phi^2)$. Therefore the statement that $C_{\bs \ell}^{\text{bias }\hat\phi}$ is an approximation  to the second order bias only refers to the fact that it marginally quantifies the impact of the second order term $\mathcal O(\phi^2)$ rather than the all order bias $\mathcal O(\phi^2) + \mathcal O(\phi^3) + \cdots$.

In contrast to  $C^{\text{var\,} \hat\phi}_{\bs \ell}$, which can be computed quickly using forward and inverse Fourier transformations, the calculation of $C^{\text{bias\,} \hat\phi}_{\bs \ell}$ appears to require explicit looping over $\bs\omega$ for each $\bs \ell$. This is problematic when $Z(\bs x)$ is observed on a high dimensional dense grid. However, there is an approximation to $C^{\text{bias\,} \hat\phi}_{\bs \ell}$ which is both fast and yields excellent numerical accuracy for frequencies $\bs \ell$ with small to moderate magnitude. The approximation is derived with a second order Taylor approximation $\bs C^{(2)}_{p,q,\bs k + \bs\omega}\approx \bs C^{(2)}_{p,q,\bs k} +  \nabla\bs C^{(2)}_{p,q,\bs k}\bs \omega +  \frac{1}{2}\bs \omega^T \nabla^2\bs C^{(2)}_{p,q,\bs k} \bs \omega$ so that
\begin{align}
 \hat\phi_{\bs \ell}\big\{{\bs C^{(2)}_{p,q,\bs k}}- {\bs C^{(2)}_{p,q,\bs k + \bs\omega}}\big\}
 \approx
 -  \hat\phi_{\bs \ell}\big\{ \nabla\bs C^{(2)}_{p,q,\bs k} \big\} \,\bs \omega
 - \textstyle\frac{1}{2}\bs \omega^T\, \hat\phi_{\bs \ell}\big\{ \nabla^2\bs C^{(2)}_{p,q,\bs k} \big\}\,\bs \omega.
 \label{eq: fast approximation to Cpbias}
\end{align}
The advantage being that $ \hat\phi_{\bs \ell}\big\{ \nabla\bs C^{(2)}_{p,q,\bs k} \big\}$ and $\hat\phi_{\bs \ell}\big\{ \nabla^2\bs C^{(2)}_{p,q,\bs k} \big\}$ only need to be computed once and can therefore be factored out of the integral (\ref{eq in intro on spec den of O(phi2)}). The factored integral is then recognized as a convolution which can be quickly computed using forward and inverse Fourier transforms.
The quality of the approximation to  $C^{\text{bias\,} \hat\phi}_{\bs \ell}$ is illustrated in Section \ref{section: nonstat example d=1} where simulations are done on a sufficiently coarse grid to allow a comparison of both $C^{\text{bias\,} \hat\phi}_{\bs \ell}$ and the fast approximation. In Section \ref{section: nonstat example d=2}, however, simulations are done on a two dimensional grid which is dense enough to necessitate the fast approximation to $C^{\text{bias\,} \hat\phi}_{\bs \ell}$.

%
%
\subsection{An illustration of the bias reduction due to local invariance}
\label{subsubsection: illustration of the bias reduction}

In this section we give an example of two stochastic processes with nearly the same values of $\bs C_{\bs k}^{(0)}, \bs C_{\bs k}^{(1)}$ and $\bs C_{\bs k}^{(2)}$, discussed above, but where one process is not locally invariant.  A quadratic estimate of nonstationarity is derived for both models and the resulting bias of each is compared. The main conclusion is that a small deviation from the locally invariant structure results in a bias that is orders of magnitude larger than what is found in the local invariant model.

Consider the following two periodic nonstationary stochastic processes\footnote{We follow our notational convention and use non-bold symbols in this section to indicate scalar quantities for $d=1$. Moreover, due to the periodic nature of $Z(t)$ and $\tilde Z(t)$, our notation dictates
$\int \equiv \sum_{k\in \Bbb Z}$ in (\ref{eq: def Z(t) and tildeZ(t)}).}
 on $[-\pi,\pi)$
\begin{align}
\label{eq: def Z(t) and tildeZ(t)}
Z(t) &:= \int  e^{i tk} e^{i\phi(t)k }\sqrt{C_k}\frac{dB_{k}}{\sqrt{2\pi}},\qquad    \widetilde Z(t) := \int    e^{i tk} e^{\phi(t)|k| }\sqrt{C_k}\frac{dB_{k}}{\sqrt{2\pi}}
\end{align}
where  $dB_{k}$ is complex Gaussian white noise, $C_k$ is the Mat\'ern spectral density  with parameters $\nu = 2, \rho = 0.025, \sigma = 1$ (using parameterization given in equation (33) of \cite{stein2012interpolation}),  $C^{\phi\phi}_k$ is the Mat\'ern spectral density with parameters $\nu = 3 , \rho = 2\pi/10, \sigma = 0.03$. Notice that $Z(t)$ has a local invariant nonstationarity, whereas $\widetilde Z(t)$ does not. Indeed the analog to expansion (\ref{first order term, intro}) for the two covariance structures is given by
\begin{align}
\text{cov}(Z(t),Z(s))= C^{(0)}(t\!-\!s) &+ \big(\phi(t)-\phi(s)\big)^{\phantom{2}} C^{(1)}(t\!-\!s) \nonumber\\
									 &+ \big(\phi(t)-\phi(s)\big)^2 C^{(2)}(t\!-\!s) + \mathcal O(\phi^3)\label{eq: cov expansion cov(Z(t),Z(s))}\\
\text{cov}(\widetilde Z(t),\widetilde Z(s)) = C^{(0)}(t\!-\!s) &+ \big(\phi(t)+\phi(s)\big)^{\phantom{2}} \widetilde C^{(1)}(t\!-\!s) \nonumber\\
															&+ \big(\phi(t)+\phi(s)\big)^2 \widetilde C^{(2)}(t\!-\!s) + \mathcal O(\phi^3)\label{eq: cov expansion cov(widetilde Z(t),widetilde Z(s))}
\end{align}
where $\widetilde C^{(1)}_k$ and $\widetilde C^{(2)}_k$ are related to the corresponding local invariant terms as follows
\begin{align}
\widetilde C^{(1)}_k &:= |C^{(1)}_k|, \qquad
\widetilde C^{(2)}_k := -C^{(2)}_k.
\label{eq: comparison of invariant and non-invariant expansions}
\end{align}

The quadratic estimate $\hat\phi_{\ell}$ based on the observed local invariant process $Z^{obs}(t)= Z(t)$, without observational noise (so that $C^{N\!N}_{\ell}\equiv 0$), is defined by (\ref{explicit form of the qe in body of paper}). To construct a quadratic estimate of $\phi_\ell$ based on observations $\widetilde Z^{obs}(t) = \widetilde Z(t)$ first notice that one can use the expansion (\ref{eq: cov expansion cov(widetilde Z(t),widetilde Z(s))}) to derive the following approximation
\begin{align*}
E\big(\widetilde Z_{k+\ell}\widetilde Z_{-k}\big)  \approx  \phi_{\ell}\left( \widetilde C^{(1)}_{k} + \widetilde C_{k+\ell}^{(1)}\right).
\end{align*}
This is similar to
(\ref{eq: first order cov intro}) with the exception of one sign change necessary to accommodate the non local invariant structure in $\widetilde Z$.
The above approximation can now be used to define the following quadratic estimate of $\phi_\ell$, denoted $\widetilde \phi_{\bs \ell}$, from observations $\widetilde Z^{obs}(t)$
\begin{align*}
    \widetilde\phi_{\ell}
    &:=  \widetilde A_{\ell}
        \int
        {\Big(\widetilde C^{{(1)}}_{k} + \widetilde C^{{(1)}}_{k+\ell}\Big)}^{\! *}
        \frac{\widetilde Z^{obs}_{k+\ell}\widetilde Z^{obs}_{-k}}{\widetilde C^{ZZobs}_{k+\ell}\widetilde C^{ZZobs}_{k}}\frac{dk}{\sqrt{2\pi}}
\end{align*}
where $\widetilde A_{\ell}$ is defined just as in (\ref{in text eq of the normalizing constant}) with the exception that the minus sign is switched to a plus sign. Moreover, the approximations given in Sections \ref{SubSection: var} and \ref{SubSection: bias} can be similarly modified---just changing the negative sign in (\ref{2nd order bias in local invariant section II}) and in the definition of $C_{\ell}^{\text{var }\hat \phi}$---to produce analogous approximations for the variance and bias of $\widetilde \phi_{\ell}$, denoted  $C_{\ell}^{\text{var }\widetilde \phi}$ and  $C_{\ell}^{\text{bias }\widetilde \phi}$ respectively.

Figure \ref{Figure 1} shows the second order bias and variance approximation for the local invariant estimate $\hat\phi_{\ell}$ (\textbf{shown at left}) compared to the non local invariant estimate $\widetilde\phi_{\ell}$ (\textbf{shown at right}).
The left plot shows  $\ell^2 C_{\ell}^{\text{var }\hat \phi}$ and $\ell^2 C_{\ell}^{\text{bias }\hat \phi}$ (\textbf{solid-green} and \textbf{dashed-red}, respectively) whereas the right plot shows  $\ell^2 C_{\ell}^{\text{var }\widetilde \phi}$ and  $\ell^2 C_{\ell}^{\text{bias }\widetilde  \phi}$ (\textbf{solid-green} and \textbf{dashed-red}, respectively).
Both plots use the same axis range and additionally show the signal spectral density $\ell^2C^{\phi\phi}_\ell$ (\textbf{dotted line}) for comparison of the respective signal to noise ratios. Note that all spectral densities shown are multiplied by $\ell^2$ to improve the visualization of the high frequency power.

\begin{figure}[H]
\includegraphics[height=6.2cm]{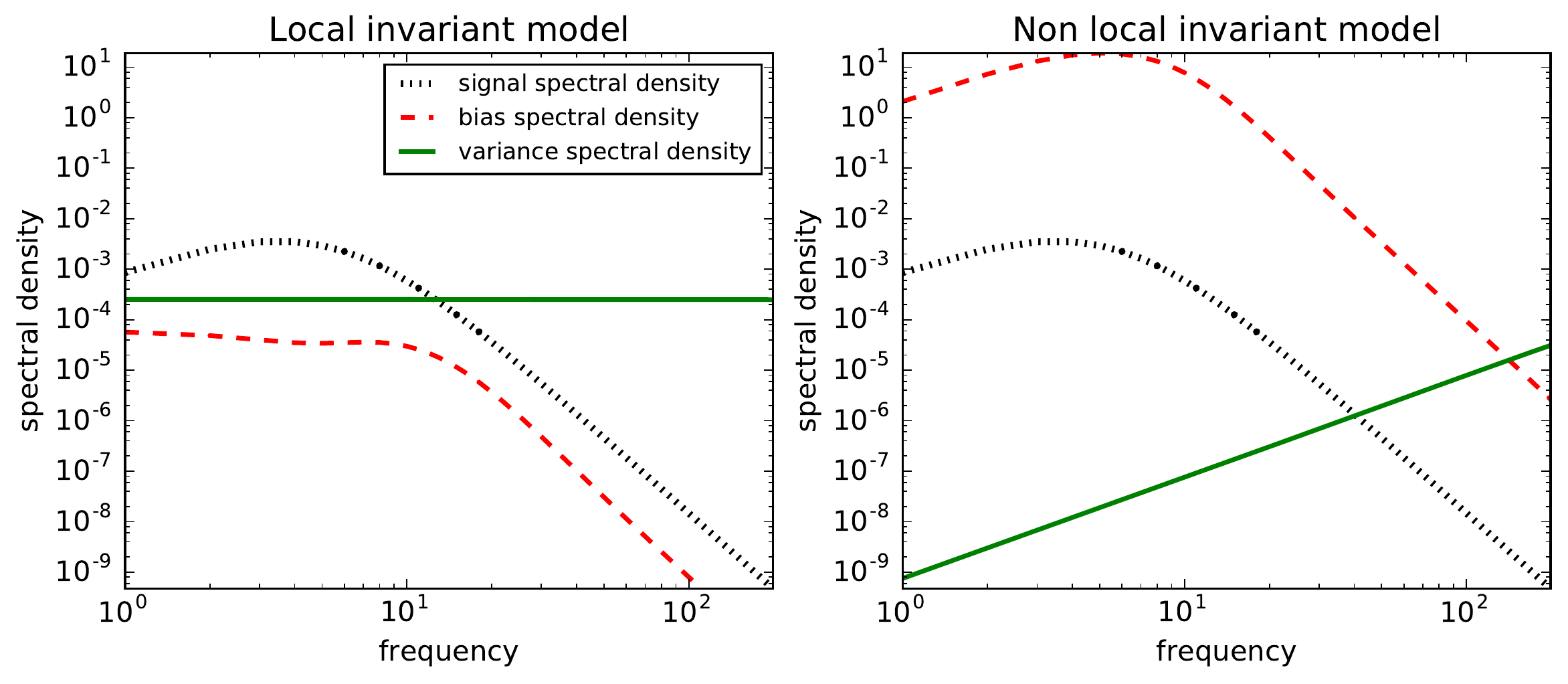}
\caption{
 These plots are intended to give an illustration  of the bias attenuation effect of the local invariant covariance structure (the specifics of the models are given in Section \ref{subsubsection: illustration of the bias reduction}).  The plot \textbf{at left} corresponds to the local invariant model and show estimation variance $\ell^2 C_{\ell}^{\text{var }\hat \phi}$ (\textbf{solid-green}), bias $\ell^2 C_{\ell}^{\text{bias }\hat \phi}$ (\textbf{dashed-red}) and the signal $\ell^2 C_{\ell}^{\phi\phi}$ (\textbf{dotted-black}). In contrast the plot \textbf{at right} shows their non local invariant counterparts.  Notice that the locally invariant second order bias is orders of magnitude smaller than the corresponding non local invariant bias, even though the functions $ C^{(1)}_k$ and $ C^{(2)}_k$ are similar to their non local invariant counterparts  $\widetilde C^{(1)}_k$ and $\widetilde C^{(2)}_k$. }
\label{Figure 1}
\end{figure}

The main conclusion from Figure \ref{Figure 1} is that, even though the functions $\widetilde C^{(1)}_k$ and $\widetilde C^{(2)}_k$ are very similar to their locally invariant counterparts, the local invariant model results in quadratic estimation bias that is orders of magnitude smaller than the corresponding non local invariant estimation bias. Indeed at most small frequencies $\ell$ one has
\begin{align*}
 C_{\ell}^{\text{bias }\hat \phi} &\ll \min(C_{\ell}^{\phi\phi}, C_{\ell}^{\text{var }\hat \phi}) \\
 C_{\ell}^{\text{bias }\widetilde \phi} &\gg \max( C_{\ell}^{\phi\phi}, C_{\ell}^{\text{var }\widetilde \phi}).
\end{align*}
Figure \ref{Figure 1} also shows that for small $\ell$, the signal to noise ratio $C_{\ell}^{\phi\phi}/C_{\ell}^{\text{var }\hat \phi}$ is large enough to suggest $\hat\phi(t)$ will be successful at tracking the large scale features of $\phi(t)$. Moreover, even at large $\ell$, where the signal to noise ratio for $\hat\phi_{\ell}$ is small, the fact that the bias is sub-dominant to the signal implies that detection of $C_{\ell}^{\phi\phi}$ is possible by averaging over a large number of frequencies to attenuate the impact of $C_{\ell}^{\text{var }\hat \phi}$.

%
%
\section{Nonstationary spectral phase model}
\label{section: NPhase}

In this section we specialize the results of the previous section to random fields which have a particular form: nonstationary spectral phase random fields. These models are locally invariant (c.f. Definition \ref{Definition 1}) and can be viewed as generalized warping models which are amenable to quadratic estimates. The general formulas for $\hat\phi_{\bs \ell}$, $C_{\bs \ell}^{\text{var }\hat\phi}$ and $C_{\bs \ell}^{\text{bias }\hat\phi}$, given in the previous section, are derived for the nonstationary spectral phase model to yield exact formulas. In subsection \ref{section: Locally attainable spectral densities} we present a characterization of the local spectral densities for nonstationary spectral phase models using the theory of optimal transport and the $L_2$-Wasserstein metric. In subsection \ref{section: modeling xi, C and eta} we present informal guidance for specifying some of the modeling parameters of the nonstationary models presented here.
Finally, in sections \ref{section: nonstat example d=1} and \ref{section: nonstat example d=2}, we
present a set of simulations which demonstrate the quadratic estimate and the accuracy of our approximation to the mean squared sampling properties.

\begin{definition}\label{def: nonstate phase}
A real random field $Z(\bs x)$ on $\Bbb R^d$ is said to be a \textbf{nonstationary spectral phase random field} if there exists  functions $C_{\bs k}:\Bbb R^d \rightarrow \Bbb R^+$, $\bs \theta(\bs x):\Bbb R^d \rightarrow \Bbb R^d$ and  $\bs\eta_{\bs k}:\Bbb R^d \rightarrow \Bbb R^d$ such that $C_{\bs k}$ has finite $L_1(\Bbb R^d)$ norm, $C_{-\bs k} = C_{\bs k}$,  $\bs \eta_{-\bs k} =-\bs\eta_{\bs k}$ and
\begin{align}
\label{eq: nonstate phase}
Z(\bs x)= \int \exp\big(i\bs x\cdot \bs k\big)\exp\big(i\bs\theta(\bs x) \cdot \bs\eta_{\bs k}\big) \,\sqrt{C_{\bs k}} \frac{dW_{\bs k}}{(2\pi)^{d/2}}
\end{align}
where $dW_{\bs k}$  denotes a complex Gaussian white noise random measure on $\Bbb R^d$ which satisfies $E|dW_{\bs k}|^2 = d\bs k$.
\end{definition}

The covariance function $C^{\bs \theta}(\bs x,\bs y):=\text{cov}(Z(\bs x), Z(\bs y))$ for the random field defined by (\ref{eq: nonstate phase}), conditioning on $\bs \eta_{\bs k}$ and $\bs \theta(\bs x)$, can be computed as follows
\begin{align}
\label{eq: NPhase cov model}
C^{\bs \theta}(\bs x,\bs y)=
\int  \exp\big(i(\bs x-\bs y)\cdot \bs k\big)\exp\big( i(\bs\theta(\bs x)-\bs\theta(\bs y)) \cdot \bs\eta_{\bs k}\big)\,C_{\bs k}\frac{d\bs k}{(2\pi)^d} .
\end{align}
The fact that $C^{\bs \theta}(\bs x,\bs y)$ can be written as a function of $\bs x-\bs y$ and $\bs\theta(\bs x)-\bs\theta(\bs y)$ implies that $Z(\bs x)$ has the local invariant property (see Definition \ref{Definition 1}) with respect to the nonstationarity characterized by $\bs\theta(\bs x)$. By assuming $\bs \eta_{\bs k}$ is known and $\bs \theta(\bs x)$ is characterized by a scalar potential $\phi(\bs x):\Bbb R^d \rightarrow \Bbb R$ (see Assumption \ref{Assumption 2}) the results of Section \ref{Section: the quad est} can be applied to generate a quadratic estimate $\hat\phi_{\bs \ell}$ based on a single realization of $Z(\bs x)$ with additive stationary noise.  Indeed, by expanding $\exp( i(\bs\theta(\bs x)-\bs\theta(\bs y)) \cdot \bs\eta_{\bs k})$ in (\ref{eq: NPhase cov model}), to second order, one obtains the following expression for the terms in (\ref{first order term, intro}) and (\ref{eq: 2nd order bias (x,y)})
\begin{align}
C_{\bs k}^{(0)} =  \frac{C_{\bs k}}{(2\pi)^{d/2}}, \qquad \bs C_{\bs k}^{(1)} = \frac{i  \bs \eta_{\bs k} C_{\bs k}}{(2\pi)^{d/2}},\qquad \bs C_{\bs k}^{(2)} = -\frac{\bs \eta_{\bs k}\bs \eta_{\bs k}^T C_{\bs k}}{2(2\pi)^{d/2}}.\label{eq: C(1), C(2) and C(3) for nonstationary phase}
\end{align}
The last ingredient needed for computing  $\hat\phi_{\bs \ell}$,\, $C_{\bs\ell}^{\text{var }\hat \phi}$ and $C_{\bs \ell}^{\text{bias }\hat \phi}$ is an expression for the marginal spectral density of the observed field $Z^{obs}(\bs x):=Z(\bs x) + N(\bs x)$, denoted $C^{ZZobs}_{\bs k}$ in Section \ref{Section: the quad est}. Notice that  Assumption \ref{Assumption 3} guarantees that $\bs \theta(\bs x)$ is a stationary mean zero Gaussian random field. Therefore
\begin{align*}
E\big(\exp(i (\bs \theta(\bs x) - \bs\theta(\bs y))\cdot \bs \eta_{\bs k})\big)
& = \exp\!\big(\!-\textstyle\frac{1}{2}\bs\eta^T_{\bs k} \bs\Sigma(\bs x-\bs y) \bs \eta_{\bs k}\big)
\end{align*}
where $\bs\Sigma(\bs x-\bs y)$ is the covariance matrix of $\bs\theta(\bs x) - \bs\theta(\bs y)$.
This implies that the marginal spectral density of the observations has the form
\begin{align}
C^{ZZobs}_{\bs k} = {(2\pi)}^{d/2}C^{Zm}_{\bs k} + C^{NN}_{\bs k} \label{eq: CZZobs for nonstationary phase}
\end{align}
where $C^{Zm}_{\bs k}$ is the Fourier transform of the marginal auto covariance of $Z(\bs x)$ and is given by
\begin{align*}
C^{Zm}(\bs x-\bs y)=
\int  \exp(i(\bs x-\bs y)\cdot \bs k)\exp\!\big(\!-{\textstyle\frac{1}{2}}\bs\eta^T_{\bs k} \bs\Sigma(\bs x-\bs y) \bs \eta_{\bs k}\big)\, C_{\bs k}\frac{d\bs k}{(2\pi)^d}.
\end{align*}
Now the expressions for $C^{ZZobs}_{\bs k}$, $\bs C_{\bs k}^{(1)}$ and $\bs C_{\bs k}^{(2)}$ given in (\ref{eq: C(1), C(2) and C(3) for nonstationary phase})  and (\ref{eq: CZZobs for nonstationary phase})
completely define the quadratic estimate $\hat\phi_{\bs \ell}$, the second order bias spectral density $C_{\bs\ell}^{\text{bias }\hat \phi}$ and the approximation to estimation variance characterized by $C_{\bs\ell}^{\text{var }\hat \phi}$ derived in Section \ref{Section: the quad est}.

\subsection{Locally attainable spectral densities}
\label{section: Locally attainable spectral densities}

In this section we investigate the set of possible local spectral densities, for different values of the nonstationary function $\bs \theta(\bs x)$, under the nonstationary spectral phase random field model of $Z(\bs x)$. Since the nonstationarity in $Z(\bs x)$ is exclusively due to local variation of a spectral phase, there is significant restriction on how local spectral densities can vary with $\bs x$. For example, one can easily see that all the local spectral densities of $Z(\bs x)$ must have the same integral (so that the pointwise variance of $Z(\bs x)$ is a constant function of $\bs x$). In what follows we characterize further restrictions and, in doing so, identify a second source of estimation bias due to the curved nature of the local spectral models. The theory of optimal transport and the $L_2$-Wasserstein metric play an important role in the characterization of local stationary models. We refer the reader to the excellent book \cite{villani2003topics} for an overview of the subject.

The local invariance property of nonstationary spectral phase models implies that the local distribution of $Z(\bs x)$ is invariant to changes in the magnitude of $\bs \theta(\bs x)$. However, the local behavior of $Z(\bs x)$ \textit{is} sensitive to the gradient of $\bs \theta(\bs x)$. In particular, suppose $\bs\theta(\bs x)$ has the form $\bs A \bs x + \bs b$  where $\bs A\in \Bbb R^{d\times d}$ and $\bs b \in \Bbb R^d$. In this case the covariance function $C^{\bs \theta}(\bs x,\bs y)$ is still invariant to changes in $\bs b \in \Bbb R^d$ but is sensitive to changes in $\bs A\in \Bbb R^{d\times d}$. Indeed assuming $\bs\theta(\bs x)=\bs A \bs x + \bs b$ one has
\begin{align}
C^{\bs \theta}(\bs x,\bs y)
&=\int  \exp\!\big(i(\bs x-\bs y)\!\cdot\! (\bs k + \bs A^T\bs\eta_{\bs k})\big) \, C_{\bs k}\frac{d\bs k}{(2\pi)^d}
=\int\exp\!\big(i(\bs x-\bs y)\!\cdot\! \bs\omega\big)  \frac{d\lambda(\bs \omega)}{(2\pi)^d} \label{eq: local spectral density for nonstationary phase}
\end{align}
where $\lambda$ is the spectral measure obtained by a change of variables $\bs\omega = \bs k + \bs A^T\bs\eta_{\bs k}$ (under appropriate measurability assumptions on $\bs \eta_{\bs k}$) Therefore when $\bs\theta(\bs x) = \bs A \bs x + \bs b$ in a local neighborhood about $\bs x$ the process $Z(\bs x)$ becomes locally stationary with local spectral measure given by $\lambda$.

The optimal transport literature uses the notation $\bs v \,\sharp\, \lambda(B) := \lambda(\bs v^{-1}(B))$ to denote the push forward of a measure $\lambda$ on $\Bbb R^d$ under a measurable transformation $\bs v_{\bs k}:\Bbb R^d\rightarrow \Bbb R^d$. For our needs it will be useful to extend this definition to spectral measures $\lambda$ which have a spectral density $C_{\bs k}$  with respect to Lebesgue measure on $\Bbb R^d$.  In particular we let $\bs v_{\bs k}\,\sharp\, C_{\bs k}$ denote the push forward of the measure  $C_{\bs k}d\bs k$ under the transformation $\bs v_{\bs k}$. This notation allows one to easily express the spectral measure $\lambda$  in (\ref{eq: local spectral density for nonstationary phase}) as
\[
\lambda= (\bs k+ \bs A^T \bs\eta_{\bs k})\,\sharp\, C_{\bs k}
\]
and, in doing so, creates a succinct notation for the collection of locally attainable spectral measures defined  as follows.

\begin{definition}
Suppose $Z(\bs x)$ is a nonstationary spectral phase random field on $\Bbb R^d$ satisfying Definition \ref{def: nonstate phase}. \textit{\textbf{The family of locally attainable spectral measures for $Z(\bs x)$}} is defined to be the collection of measures
\[
\mathscr C^{C,\bs\eta}:= \big\{ (\bs k+ \bs A^T \bs\eta_{\bs k})\,\sharp\, C_{\bs k}: \bs A \in \Bbb R^{d\times d}\big\}
\]
so that for each $\lambda \in \mathscr C^{C,\bs\eta}$ there exists a matrix $\bs A$ such that when $\bs \theta(\bs x) = \bs A\bs x$ the random field $Z(\bs x)$ becomes stationary with spectral measure $\lambda$.
\end{definition}

Notice that, depending on $\bs \xi$, there may a restriction on the possible matrices $\bs A$ which can satisfy $\bs \theta(\bs x) = \bs A\bs x$. This will further limit the set of attainable local spectral densities but is not included in the definition of $\mathscr C^{C,\bs \eta}$. The role of $\bs \xi$, in terms of modeling $Z(\bs x)$, is discussed in Section \ref{section: modeling xi, C and eta} below.

The following claim shows that given any two spectral densities $C_{\bs k}$ and $\tilde C_{\bs k}$ with finite second moments and the same $L_1(\Bbb R^d)$ integral, there exists a nonstationary spectral phase random field $
Z(\bs x)$ which has both $C_{\bs k}$ and $\tilde C_{\bs k}$ as locally attainable spectral densities (i.e. $C_{\bs k},\tilde C_{\bs k}\in\mathscr C^{C,\bs \eta}$). Moreover, each measure contained in the $L_2$-Wasserstein geodesic connecting $C_{\bs k}d\bs k$ to $\tilde C_{\bs k}d\bs k$ is also locally attainable by $Z(\bs x)$.

\begin{claim}[\textbf{Any pair of spectral densities with the same integral are attainable}]
\label{claim any two spec are attainable}
Let $d\geq 1$ be an integer, $t_0 > 0$ be a real number and $C_{\bs k}$, $\tilde C_{\bs k}$ be two spectral densities on $\Bbb R^d$ with finite second moments such that $\sigma^2 = \int_{\Bbb R^d}  C_{\bs k}d\bs k = \int_{\Bbb R^d} \tilde C_{\bs k}d\bs k$. Then there exists a vector field $\bs \eta_{\bs k}\colon \Bbb R^d \rightarrow \Bbb R^d$ which is $L_2(\Bbb R^d)$ integrable with respect to $C_{\bs k}d\bs k$ and  generates a one dimensional curve of spectral measures  $\{\lambda^{(t)}: t\in [0,t_0]\}$, defined by
\begin{align}
\lambda^{(t)} := (\bs k + t\bs \eta_{\bs k})\,\sharp\, C_{\bs k}
\end{align}
with endpoints $\lambda^{(0)}=C_{\bs k}d\bs k$ and $\lambda^{(t_0)}=\tilde C_{\bs k}d\bs k$,
such that  $\{\lambda^{(t)}:t\in[0,t_0] \}\subset \mathscr C^{C,\bs\eta}$. In particular there exists a nonstationary spectral phase random field model for which  $C_{\bs k}$ and $\tilde C_{\bs k}$ are both locally attainable.
Moreover,  $\{\lambda^{(t)}: t\in [0,t_0]\}$ is a $L_2$-Wasserstein geodesic path within the class of absolutely continuous spectral measures (with total mass $\sigma^2$ and finite second moments) and each measure $\lambda^{(t)}$ has a density $C^{(t)}_{\bs k}$ with respect to Lebesgue measure on $\Bbb R^d$ which (weakly) satisfies
\begin{align}
\label{div flow for geo}
\partial_t C^{(t)}_{\bs k} +\text{\rm div}(\bs \eta^{(t)}_{\bs k}C^{(t)}_{\bs k}) = 0
\end{align}
for all $t\in [0,t_0]$ where $\bs \eta^{(t)}_{\bs k}:=\bs \eta_{T_t^{-1}(\bs k)}$ and $T_t(\bs k):= \bs k + t\bs \eta_{\bs k}$.
\end{claim}
\begin{proof}
By standard optimal transport theory (see \cite{villani2003topics} for example) the assumptions on $C_{\bs k}$ and $\tilde C_{\bs k}$ guarantee the existence of a convex function $\psi_{\bs k}:\Bbb R^d \rightarrow \Bbb R$ such that $\nabla \psi_{\bs k}$ is the optimal transport from $C_{\bs k}d\bs k$ to $\tilde C_{\bs k}d\bs k$.  Let $\bs \eta_{\bs k}:= \frac{1}{t_0}(\nabla \psi_{\bs k}-\bs k)$ so that
\begin{align}
\label{geo eq number 1}
\lambda^{(t)}:=(\bs k + t\bs\eta_{\bs k})\,\sharp\, C_{\bs k} = \big((1-\textstyle\frac{t}{t_0})\bs k + \textstyle\frac{t}{t_0}\nabla\psi_{\bs k}\big)\,\sharp\, C_{\bs k}.
\end{align}
The particular form of the right hand side of (\ref{geo eq number 1})  implies each measure $\lambda^{(t)}$ has a density $C_{\bs k}^{(t)}$ with respect to Lebesgue measure and the path of measures $\{\lambda^{(t)}\colon t\in[0,t_0] \}$ forms an $L_2$-Wasserstein geodesic with endpoints $C_{\bs k}d\bs k$ and $\tilde C_{\bs k}d\bs k$ at $t=0$ and $t=t_0$ respectively (see Proposition 5.9 in \cite{villani2003topics}). Moreover, $C^{(t)}_{\bs k}$ weakly satisfies (\ref{div flow for geo}) by Theorem 5.34 of \cite{villani2003topics} and the fact that $C^{(t)}_{\bs k} = T_t\,\sharp\, C^{(0)}_{\bs k}$ where $\partial_t T_t(\bs k) = \bs \eta_{\bs k}$.
By  setting $\bs A=t\bs I_d$ in (\ref{eq: local spectral density for nonstationary phase}) one has
\[
\int  \exp\!\big(i(\bs x-\bs y)\!\cdot\! (\bs k + t\bs\eta_{\bs k})\big)\, C_{\bs k}\frac{d\bs k}{(2\pi)^d}
=\int \exp\!\big(i(\bs x-\bs y)\!\cdot\! \bs\omega\big)\,  C^{(t)}_{\bs \omega}\frac{d\bs \omega}{(2\pi)^d}
\]
which implies that for each $t\in[0,t_0]$ the measure $\lambda^{(t)}$ is a locally attainable spectral measure.
\end{proof}

Isotropic spectral densities are an important special case for many statistical applications. The following claim allows considerable simplification for the construction of the vector field $\bs \eta_{\bs k}$ guaranteed by Claim \ref{claim any two spec are attainable}.

\begin{claim}[\textbf{Special case for isotropic spectral densities}]
\label{claim: isotropic spectral densities}
Let $d\geq 1$ be an integer, $t_0 > 0$ be a real number and $C_{|\bs k|}$, $\tilde C_{|\bs k|}$ be two isotropic spectral densities on $\Bbb R^d$ with finite second moments and total mass $\sigma^2$. Define
\begin{align}
\bs \eta_{\bs k}:= \frac{1}{t_0}\Big(\tilde F^{-1}\circ F(|\bs k|)\frac{\bs k}{|\bs k|}-\bs k\Big)
\end{align}
 where $F(r) := \frac{2\pi^{d/2}}{\sigma^2\Gamma(d/2)}\int_0^r \xi^{d-1}C_\xi d\xi$ and $\tilde F(r) := \frac{2\pi^{d/2}}{\sigma^2\Gamma(d/2)}\int_0^r \xi^{d-1}\tilde C_\xi d\xi$. Then for all $t\in[0,t_0]$, $\bs \eta_{\bs k}$ generates the spectral measures $\lambda^{(t)}:= (\bs k + t\bs \eta_{\bs k})\,\sharp\, C_{|\bs k|} $ defined in Claim \ref{claim any two spec are attainable}. In particular, $\{\lambda^{(t)} \colon t\in[0,t_0]\}$ forms a $L_2$-Wasserstein geodesic path of locally attainable spectral densities in $\mathscr C^{C,\bs \eta}$, with endpoints $\lambda^{(0)}=C_{|\bs k|}d\bs k$ and $\lambda^{(t_0)}=\tilde C_{|\bs k|}d\bs k$, where
$\bs k + t\bs \eta_{\bs k}$ is the optimal transport from $\lambda^{(0)}$ to $\lambda^{(t)}$.
\end{claim}

\begin{proof}
By the proof of Claim \ref{claim any two spec are attainable} it will be sufficient to show that $\tilde F^{-1}\circ F(|\bs k|)\frac{\bs k}{|\bs k|}$ is the optimal transport from $C_{|\bs k|}$ to $\tilde C_{|\bs k|}$.
Let $\bs K$ and $\tilde {\bs K}$ be random vectors in $\Bbb R^d$ with densities $C_{|\bs k|}/\sigma^2$ and $\tilde C_{|\bs k|}/\sigma^2$, respectively. By the distributional rotational symmetry of $\bs K$ there exists a convex $\psi(r):\Bbb R^+ \rightarrow \Bbb R$ such that $\nabla (\psi(|\bs k|))=\psi^\prime(|\bs k|)\frac{\bs k}{|\bs k|}$ is the optimal transport from $\mathscr L\bs K$ to $\mathscr L\tilde{\bs K}$ (where $\mathscr L\bs K$ and $\mathscr L\tilde{\bs K}$ denotes the law, i.e. probability distribution, of $\bs K$). Also notice that $\psi^\prime(r)$ is the optimal transport from $\mathscr L |\bs K|$ to  $\mathscr L|\tilde{\bs K}|$ since $\psi(r)$ is convex and
\begin{align*}
P(\psi^\prime(|\bs K|)\leq r) = P\big(\big|\psi^\prime(|\bs K|)\textstyle\frac{\bs K}{|\bs K|}\big|\leq r\big) = P(|\tilde{\bs K}|\leq r).
\end{align*}
The optimal transport between two univariate random variables (see \cite{villani2003topics}) is given by the composition of the quantile function (of the target measure) and the cumulative distribution function (of the base measure). Therefore the optimal transport from $\mathscr L|\bs K|$ to $\mathscr L|\tilde{\bs K}|$ is given by $\tilde F^{-1}\circ F(r)$ where $F(r)=\frac{2\pi^{d/2}}{\sigma^2\Gamma(d/2)}\int_0^r \xi^{d-1}C_\xi d\xi$ and $\tilde F(r)=\frac{2\pi^{d/2}}{\sigma^2\Gamma(d/2)}\int_0^r \xi^{d-1}\tilde C_\xi d\xi$ are the cumulative distribution functions of $|\bs K|$ and $|\tilde{\bs K}|$, respectively. By the uniqueness of optimal transports one has $\psi^\prime(r)=\tilde F^{-1}\circ F(r)$ and therefore
\[
\nabla (\psi(|\bs k|))=\psi^\prime(|\bs k|)\frac{\bs k}{|\bs k|} = \tilde F^{-1}\circ F(|\bs k|)\frac{\bs k}{|\bs k|}
\]
is the optimal transport from $\mathscr L\bs K$ to $\mathscr L\tilde{\bs K}$,
as was to be shown.
\end{proof}

As a corollary to the above theorem one can obtain partial closed form solutions for $\bs \eta_{\bs k}$ when the spectral densities $C_{\bs k}$ and $\tilde C_{\bs k}$ are both Mat\'ern spectral densities with the same integral and with finite second moments. The form of $\bs \eta_{\bs k}$, in this case, can be computed using the incomplete beta function and the quantiles of beta random variables (which is not technically given in closed form but for which simple Newton method characterizations are guaranteed to converge, see \cite{giner2014monotonically}).

\begin{corollary}[\textbf{Optimal transports between Mat\'ern spectral densities}]
\label{corollary: Optimal transports between Matern spectral densities}
Let $d\geq 1$ be an integer and $t_0,\nu,\tilde \nu, \rho, \tilde \rho,\sigma^2 > 0$ be real numbers such that $\nu,\tilde \nu >1$.
If
\begin{align}
C_{|\bs k|} &= \sigma^22^{d}\pi^{d/2}\frac{\Gamma(\nu + d/2)}{\Gamma(\nu)}\left(\frac{4\nu}{\rho^2}\right)^{\nu}\left(\frac{4\nu}{\rho^2}+|\bs k|^2\right)^{-\nu-d/2} \label{eq: Ck matern}\\
\tilde C_{|\bs k|} &= \sigma^2 2^{d}\pi^{d/2}\frac{\Gamma(\tilde\nu + d/2)}{\Gamma(\tilde\nu)}\left(\frac{4\tilde\nu}{\tilde\rho^2}\right)^{\tilde\nu}\left(\frac{4\tilde\nu}{\tilde\rho^2}+|\bs k|^2\right)^{-\tilde\nu-d/2}\label{eq: tildeCk matern}
\end{align}
 then $F$ and $\tilde F^{-1}$, defined in Claim \ref{claim: isotropic spectral densities}, are given by
\begin{align}
\label{F0 for matern}
F(r) &= I_{r^2/(4\nu / \rho^2 + r^2)}(d/2, \nu) \\
\tilde F^{-1}(u) &= \left(\frac{4\tilde\nu}{\tilde\rho^2}\right)^{1/2}\left(\frac{1}{Q_u(d/2,\tilde\nu)} -1 \right)^{-1/2}
\label{Ft0inv for matern}
\end{align}
where $I_x(p,q)$ is the incomplete beta function and $Q_u(p,q)$ is the quantile function for a univariate $\text{Beta}(p,q)$ random variable evaluated at $u\in (0,1)$.
\end{corollary}
\begin{proof}
First notice that $ (2\pi)^d\sigma^2 =\int_{\Bbb R^d} C_{|\bs k|}d\bs k =  \int_{\Bbb R^d} \tilde C_{|\bs k|}d\bs k$ and the constraints $\nu, \tilde \nu>1$ are sufficient to ensure  $C_{|\bs k|}$ and $\tilde C_{|\bs k|}$ have finite second moments. Therefore Claim \ref{claim: isotropic spectral densities} applies. For any $\nu,a>0$, the change of variables $x= y^2/(a+y^2)$ gives
\begin{align*}
\int_0^r \frac{y^{d-1}}{(a+y^2)^{\nu+d/2}} dy = \frac{1}{2a^\nu}\int_0^{\frac{r^2}{a+r^2}} (1-x)^{\nu-1}x^{d/2-1}dx = \frac{B(d/2,\nu)}{2a^\nu}I_{r^2/(a + r^2)}(d/2,\nu).
\end{align*}
Therefore $F(r) := \frac{2\pi^{d/2}}{\sigma^2(2\pi)^d\Gamma(d/2)}\int_0^r \xi^{d-1}C_\xi d\xi = I_{r^2/(4\nu / \rho^2 + r^2)}(d/2,\nu) $
and similarly for $\tilde F(r)$.
This immediately gives (\ref{F0 for matern}) an (\ref{Ft0inv for matern}).
\end{proof}

The Wasserstein structure of the locally attainable spectral models gives a convenient geometric picture for  potential difficulties when estimating local spectra in the nonstationary phase model.
For example, the space of probability distributions (with finite second moments) has positive curvature (in the sense of Aleksandrov's notion of metric curvature) under the $L_2$-Wasserstein metric \cite{ambrosio2008gradient}.
A less precise mathematical illustration of this is the fact is that the two locally attainable spectra, $(\bs k+ \bs A^T \bs\eta_{\bs k})\,\sharp\, C_{\bs k}$ and $(\bs k- \bs A^T \bs\eta_{\bs k})\,\sharp\, C_{\bs k}$, become asymptotically indistinguishable as the entries $\bs A$ become arbitrarily large. Indeed, the probability distributions of $\bs K + \bs A^T\bs\eta_{\bs K}$ and $\bs K - \bs A^T\bs\eta_{\bs K}$ are similar when $\bs K$ is a random vector with unnormalized density $C_{\bs k}$ and the magnitude of the entries of $\bs A$ are large (since $\bs\eta_{\bs k}$ has odd symmetry and $C_{\bs k}$ has even symmetric). One implication of this asymptotic non-identifiability is that estimates of the local spectra can break down when $\bs \theta(\bs x)$ has large local linear fluctuations (i.e. when the entries of $\bs A$ are large) so that the two local models $\bs \theta(\bs x)=\bs A\bs x + \bs b$ and $\bs \theta(\bs x)=-\bs A\bs x + \bs b$ are difficult to distinguish from the data.

A possible numerical tool for probing this breakdown is to analyze the so called \textbf{\textit{cut locus}} associated with $L_2$-Wasserstein geodesics. Informally, the cut locus quantifies the maximal extent of geodesic paths emanating from some point. The cut locus of the geodesics emanating from $C_{\bs k}d\bs k$ can serve as a proxy for when curvature effects will corrupt local estimation of $\bs\theta(\bs x)$ due to the similarity between the two local spectra characterized by the local linear models $\bs \theta(\bs x)=\bs A\bs x + \bs b$ and $\bs \theta(\bs x)=-\bs A\bs x + \bs b$.  If the cut locus starting at $C_{\bs k}d\bs k$ is far from $C_{\bs k}d\bs k$ this implies the geodesic paths are long and the entries of $\bs A$ need to be much larger for curvature difficulties to arise.  In particular, fix $\bs A$ and consider the  locally attainable models $(\bs k+ c\bs A^T \bs\eta_{\bs k})\,\sharp\, C_{\bs k}$ indexed by $c\in \Bbb R$. If there exists a maximal $c_0>0$ such that $\{(\bs k+ c\bs A^T \bs\eta_{\bs k})\,\sharp\, C_{\bs k}\colon c\in[0,c_0]\}$ and $\{(\bs k+ c\bs A^T \bs\eta_{\bs k})\,\sharp\, C_{\bs k}\colon c\in[-c_0,0]\}$ are both $L_2$-Wasserstein geodesic, then nonstationary local linear models of the form $\theta(\bs x) = \pm c\bs A\bs x + \bs b$ are less exposed to curvature effects when $|c|\ll c_0$.
Claim \ref{claim: cut locus}, below, allows one to numerically compute the maximal such cutoff $c_0$ which characterizes, what we call, the {\it symmetric two sided cut locus}.

\begin{claim}
\label{claim: cut locus}
Let $d\geq 1$ be an integer, $c_0 > 0$ be a real number, $\bs A\in \Bbb R^d$, $C_{\bs k}$ is a spectral density on $\Bbb R^d$ with finite second moments and $\bs \eta_{\bs k}\colon \Bbb R^d \rightarrow \Bbb R^d$ which is $L_2(\Bbb R^d)$ integrable with respect to $C_{\bs k}d\bs k$.
Suppose both
$\bs k - c_0\bs A^T\bs \eta_{\bs k}$ and $\bs k + c_0\bs A^T\bs \eta_{\bs k}$
are $C^1$ diffeomorphisms which are gradients of convex functions. Then
$\{(\bs k - c\bs A^T\bs \eta_{\bs k})\,\sharp\,C_{\bs k}\colon c\in[0,c_0]\}$ and $\{(\bs k + c\bs A^T\bs \eta_{\bs k})\,\sharp\,C_{\bs k}\colon c\in[0,c_0]\}$ are paths of absolutely continuous measures which are also $L_2$-Wasserstein geodesics.
\end{claim}
\begin{proof}
Let $\phi_{\bs k}$ and $\psi_{\bs k}$ be convex functions defined on $\Bbb R^d$ such that $\nabla \phi_{\bs k}=\bs k - c_0\bs A^T\bs \eta_{\bs k}$ and $\nabla \psi_{\bs k}=\bs k + c_0\bs A^T\bs \eta_{\bs k}$. By the diffeomorphic assumption on $\bs k - c_0\bs A^T\bs \eta_{\bs k}$ and $\bs k + c_0\bs A^T\bs \eta_{\bs k}$ there exists two spectral densities $C^{(-c_0)}_{\bs k}$ and $C^{(c_0)}_{\bs k}$ which satisfy
\begin{align*}
C^{(-c_0)}_{\bs k}d\bs k = (\bs k - c_0\bs A^T\bs \eta_{\bs k})\,\sharp\, C_{\bs k} \\
C^{(c_0)}_{\bs k}d\bs k  = (\bs k + c_0\bs A^T\bs \eta_{\bs k})\,\sharp\, C_{\bs k} .
\end{align*}
Now for any $c\in[0,c_0]$ one has
\begin{align}
\bs k - c\bs A^T\bs\eta_{\bs k} &= \big(1-\textstyle\frac{c}{c_0}\big)\bs k + \textstyle\frac{c}{c_0}\nabla\phi_{\bs k}\label{cutlocus eq 1} \\
\bs k + c\bs A^T\bs\eta_{\bs k} &= \big(1-\textstyle\frac{c}{c_0}\big)\bs k + \textstyle\frac{c}{c_0}\nabla\psi_{\bs k}\label{cutlocus eq 2}.
\end{align}
The explicit form of the right hand side of (\ref{cutlocus eq 1}) and (\ref{cutlocus eq 2}) imply $\{ (\bs k - c\bs A^T\bs \eta_{\bs k})\,\sharp\, C_{\bs k} :c\in[0,c_0]\}$  and  $\{ (\bs k + c\bs A^T\bs \eta_{\bs k})\,\sharp\, C_{\bs k} :c\in[0,c_0]\}$  are $L_2$-Wasserstein geodesics of absolutely continuous measures, connecting $C^{(-c_0)}_{\bs k}d\bs k$ to $C_{\bs k}d\bs k $ and  $C_{\bs k}d\bs k$ to $C^{(c_0)}_{\bs k} d\bs k$ respectively (by Proposition 5.9 in \cite{villani2003topics}).
\end{proof}

\subsection{Modeling $\bs \xi_{\bs k}, C_{\bs k}$ and $\bs\eta_{\bs k}$}
\label{section: modeling xi, C and eta}

In this section we give some informal guidance for specifying $\bs \xi_{\bs k}, C_{\bs k}$ and $\bs\eta_{\bs k}$ in the nonstationary spectral phase model (see Definition \ref{def: nonstate phase}). Since the quadratic estimate is  adept at detecting small departures from stationarity---partly due to the accuracy of the variance calculations and the speed at which the quadratic estimate can be simulated under a null stationary model---we focus on the situation where the statistician wants to estimate or detect nonstationary extensions of a stationary model. Further details of this approach can be found in the simulation sections \ref{section: nonstat example d=1} and \ref{section: nonstat example d=2}.

\paragraph{Modeling $C_{\bs k}$ and $\bs\eta_{\bs k}$.} In sections \ref{section: nonstat example d=1} and \ref{section: nonstat example d=2}  we model $\bs\eta_{\bs k}$ implicitly by specifying two spectral densities $C_{\bs k}, \tilde C_{\bs k}$ and require that they both be locally attainable in the nonstationary random field model (note: $C_{\bs k}$ and $\tilde C_{\bs k}$ must have the same $L_1(\Bbb R^d)$ integral). In other words we construct a vector field $\bs \eta_{\bs k}$ from $C_{\bs k}$ and $\tilde C_{\bs k}$ by requiring  $\tilde C_{\bs k}\in \mathscr C^{C,\bs \eta}$. The results of Section \ref{section: Locally attainable spectral densities} show this is possible by setting
\begin{align}
\label{eq: modeling eta}
\bs \eta_{\bs k}:=  \frac{1}{t_0}(\nabla \psi_{\bs k}-\bs k)
\end{align}
 where  $\nabla \psi_{\bs k}$ is the optimal transport from $C_{\bs k}$ to $\tilde C_{\bs k}$ and $t_0>0$ is determined  by the  desired physical units of $\bs \theta(\bs x)$ or is set to balance the bias and variance of the quadratic estimate (more on this in the next paragraph). Corollary \ref{corollary: Optimal transports between Matern spectral densities} seems particularity useful for this approach in that $C_{\bs k}$ can be determined by an overall Mat\'ern fit and $\tilde C_{\bs k}$ can be defined by perturbing the Mat\'ern  parameters in a direction of interest. For example, consider the case where one is interested in detecting nonstationarity arising from spatial variation in the Mat\'ern smoothness parameter $\nu$.  Using the notation given in Corollary \ref{corollary: Optimal transports between Matern spectral densities} one could estimate $\sigma^2$, $\rho$ and  $\nu$ (the parameters of $C_{\bs k}$) by an overall stationary fit, then define  $\tilde\nu := \nu + \epsilon$ and $\tilde \rho := \rho$ (the parameters of $\tilde C_{\bs k}$) for some $\epsilon\in\Bbb R$.

Generally larger values of $t_0$ or smaller values of $\epsilon$ will increase estimation variance and decrease estimation bias. We do not yet have a coherent story for the precise nature the dependence of bias and variance as a function of $t_0$ and $\epsilon$. However,  the cut locus of the $L_2$-Wasserstein geodesics emanating from $C_{\bs k}$  (discussed at the end of Section \ref{section: Locally attainable spectral densities}) can be a useful tool for probing this dependence. For example, given $t_0$ and $\epsilon$ one can compute the maximal $c_0$ which satisfies the antecedent of Claim \ref{claim: cut locus}. This maximal $c_0$ effectively characterizes the \textit{symmetric two-sided cut locus} (c.f. Claim \ref{claim: cut locus}) and heuristically serves to characterize an upper bound on the magnitude of the entries of $\bs A$, beyond which bias is likely to dominate.
This will be explored in more detail in Section \ref{section: nonstat example d=1} as a diagnostic tool for determining values of  $t_0$ and $\epsilon$ that result in  large quadratic estimation bias.

\paragraph{Modeling $\bs \xi_{\bs k}$.} In some cases the spectral multiplier $\bs \xi_{\bs k}$ will be constrained by the physics of a particular application. An example of such a constraint is that $\bs \theta(\bs x)$ be required to be curl free or divergence free. Indeed, a curl free constraint is enforced in the gravitational lensing problem  by setting $\bs \xi_{\bs k}=i\bs k$. In the absence of such constraints one can potentially use $\bs \xi_{\bs k}$ to restrict the possible matrices $\bs A$ which parameterize  the locally attainable spectral models  $\mathscr C^{C,\bs\eta}= \big\{ (\bs k+ \bs A^T \bs\eta_{\bs k})\,\sharp\, C_{\bs k}: \bs A \in \Bbb R^{d\times d}\big\}$.   If $\xi_{\bs k}:= i\bs k$, for example, then $\bs A$ must be of the form $(\partial_{\bs x_p}\partial_{\bs x_q}\phi(\bs x))_{p,q=1}^d = \bs U\bs \Lambda \bs U^T$ where $\bs U$ is a rotation matrix and $\bs \Lambda$ is a diagonal matrix with real entries.

\subsection{Nonstationary phase example $d=1$}

\label{section: nonstat example d=1}

In this section we present a simulation example to illustrate the quadratic estimate of $\bs \theta(\bs x)$, or equivalently the potential $\phi(\bs x)$, when observing a single realization of a nonstationary spectral phase random field $Z(\bs x)$ in dimension $d=1$ (c.f. Definition \ref{def: nonstate phase}).  According to our notational conventions, vector quantities such as  $\bs x, \bs k, \bs \xi_{\bs k},  \bs\eta_{\bs k}, \bs \theta(\bs x)$ are replaced with non-bold scalar notation $x, k, \xi_{k},  \eta_{k}, \theta(x)$ to indicate scalar quantities for $d=1$. There are multiple points we hope to convey with this example. The first is that the quadratic estimate $\hat\phi_\ell$, constructed to detect a spatially varying smoothness parameter, is fast and accurate. The second point is that $C^{\text{var\,}\hat\phi}_{\ell}$ and $C^{\text{bias\,}\hat\phi}_{\ell}$  accurately quantify the empirical variance and bias of $\hat\phi_\ell$. A third point is that the fast approximation to $C^{\text{bias\,}\hat\phi}_{\ell}$, discussed in the last paragraph of Section \ref{SubSection: bias}, is accurate over a wide range of wave numbers. Finally we illustrate qualitative features of the estimation bias which results when $\theta^\prime(x)$ is large enough to exceed the symmetric two-sided cut locus discussed in Section \ref{section: modeling xi, C and eta} (c.f. Claim \ref{claim: cut locus}).

Following the modeling approach outlined in Section \ref{section: modeling xi, C and eta} we first define $C_k$ to be the  Mat\'ern spectral density given in (\ref{eq: Ck matern}) with parameters $(\nu, \rho, \sigma^2):=(2, 0.05, 1)$. Now  $\eta_k$ is defined implicitly by specifying a second spectral density $\tilde C_k$ that is required to be locally attainable within the same nonstationary phase model for $Z(x)$. $\tilde C_k$ is defined to be the  Mat\'ern spectral density given in (\ref{eq: tildeCk matern})  with parameters $(\tilde \nu, \tilde\rho, \sigma^2):=(2.1, 0.05, 1)$. The variance parameter $\sigma^2$ is the same for both $C_k$ and  $\tilde C_k$ as per the necessary requirement for locally attainable spectral densities. Notice that the only difference between the two Mat\'ern models is the fractional smoothness parameter which is set to model nonstationarity in the local smoothness in $Z(x)$.  The parameter $t_0$ used in (\ref{eq: modeling eta}) to determine $\eta_k$ is set to $1.5$ for Figure \ref{Figure 2 zx and var} and $1.5/7$ for Figure \ref{Figure 2 zx and var, extra bias}.

A single ground truth potential $\phi(x)$ is used throughout this section and was simulated from a mean zero stationary Gaussian process with Mat\'ern parameters $(\nu, \rho, \sigma^2):=(5, 1.5, 15^2 / (2\pi)^4)$. The spectral multiplier $\xi_k$ is set to $ik$ so that $\theta(x) = \phi^\prime(x) $.  The derivative $\theta^\prime(x)$, in particular $\phi^{\prime\prime}(x)$, is shown in blue in the middle plot of figures \ref{Figure 2 zx and var} and \ref{Figure 2 zx and var, extra bias}.
The quantities $C^{(0)}_\ell, C^{(1)}_\ell, C^{(2)}_\ell$ and $C^{ZZobs}_\ell$ used to generate  $\hat\phi_\ell$, $C^{\text{var\,}\hat\phi}_{\ell}$ and $C^{\text{bias\,}\hat\phi}_{\ell}$ are determined by (\ref{eq: C(1), C(2) and C(3) for nonstationary phase}) and (\ref{eq: CZZobs for nonstationary phase}). Finally, to avoid potential aliasing issues in the simulated data $Z^{obs}(x)$, the quadratic estimate is set to ignore $10\%$ of Fourier coefficients which are nearest the Nyquist limit by truncating the weights corresponding to those frequency pairs.

\begin{figure}
\includegraphics[height=15cm]{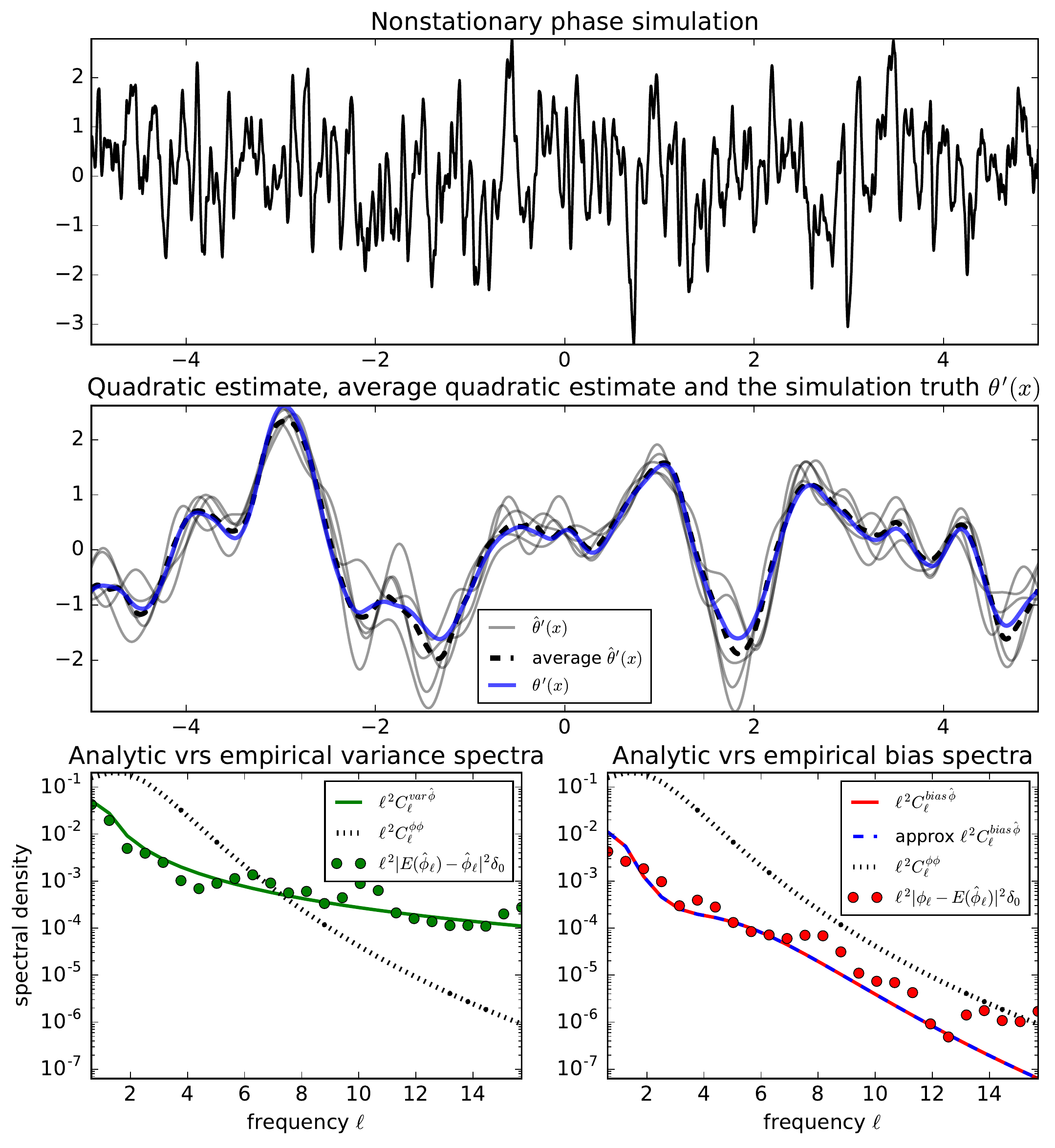}%
\caption{ This figure shows a simulation of a nonstationary spectral phase random field $Z(x)$ (\textbf{top}), the quadratic estimate of  $\theta^\prime(x)$ (\textbf{middle}) and the spectral characterizations of estimation variance (\textbf{bottom-left}) and estimation bias (\textbf{bottom-right}). The details of the simulation are given in Section \ref{section: nonstat example d=1}. In the \textbf{middle plot} the \textbf{blue line} shows the value of $\theta^\prime(x)=\phi^{\prime\prime}(x)$ which characterizes the nonstationarity in $Z(x)$ through Definition \ref{def: nonstate phase}, the  \textbf{grey lines} show different quadratic estimates $\hat\theta^\prime(x)$ each one applied to an independent realizations of $Z(x)$ with the same $\theta(x)$ and the \textbf{dashed line} shows the empirical average of $\hat\theta^\prime(x)$ over 100 such realizations. In the \textbf{bottom-left plot}
The main purpose of this simulation is intended to illustrate the accuracy of the quadratic estimate and the ability of $C^{\text{var\,}\hat\phi}_{\bs \ell}$ and $C^{\text{bias\,}\hat\phi}_{\bs \ell}$ to approximate the empirical variance and bias of the estimate. A secondary goal of this figure is to also show that the fast approximation to $C^{\text{bias\,}\hat\phi}_{\bs \ell}$ is very accurate over a wide range of small wave numbers.}
\label{Figure 2 zx and var}
\end{figure}

The process $Z(x)$ in this section is defined on $[-5,5)$ with periodic boundary conditions.
The observed process $Z^{obs}(x)$ is simulated without additive noise on $10^4$ evenly spaced observation locations in $[-5,5)$.  A simple discrete Riemann sum approximation, at each observed $x$, was used to approximate to the integral (\ref{eq: nonstate phase}) to generate the simulation of $Z^{obs}(x)$. In general, this type of approximation will result in aliasing errors. Generating a distributionally exact simulation of $Z(x)$, without any approximation, appears to be an open problem.   It is not yet clear what impact the aliasing errors, present in our simulation, have on the quadratic estimate. However, we found little empirical difference in the performance of the quadratic estimate when reducing the aliasing errors by increasing the frequency upper limit used for the Riemann sum approximation.

Figure \ref{Figure 2 zx and var} shows the results of our simulation when $t_0$ is set to $1.5$. The \textbf{top plot} shows a simulation of the nonstationary phase process $Z(x)$. The \textbf{blue line} in the \textbf{middle plot} shows $\theta^\prime(x)$ along with $5$ realizations of the quadratic estimate $\hat\theta^\prime(x)$, shown in \textbf{grey}, each one applied to an independent realization of $Z(x)$ with the same $\phi(x)$ . The \textbf{dashed line} in the \textbf{middle plot} shows an empirical estimate to $E(\hat\theta^\prime(x)|\phi)$ based on averaging the quadratic estimate applied to $100$ independent realizations $Z(x)$ all simulated with the same nonstationary potential $\phi(x)$. On average, computing these $100$ quadratic estimates (each based on $10^4$ observations) took $0.008$ seconds on a 2013 MacBook Pro with a 2.3 GHz Intel Core i7 CPU. This illustrates that the quadratic estimate can be computed extremely fast on a dense set of observations. Notice also the estimate is accurate with respect to both variance and bias. Indeed, by comparing signal spectral density $\ell^2 C_\ell^{\phi\phi}$ (\textbf{dotted black line} in \textbf{both bottom plots})  with  $\ell^2 C^{\text{var\,}\hat\phi}_{\ell}$ and $\ell^2 C^{\text{bias\,}\hat\phi}_{\ell}$  (\textbf{green and red lines} respectively) one can see that the signal-to-noise ratio for estimation accuracy per-frequency is significantly greater than $1$ for a large range of wavenumbers. The \textbf{bottom two plots} in Figure \ref{Figure 2 zx and var} show the accuracy of the analytic approximations $\ell^2 C^{\text{var\,}\hat\phi}_{\ell}$ and  $\ell^2 C^{\text{bias\,}\hat\phi}_{\ell}$ for quantifying the empirical variance and bias (\textbf{green and red dots} respectively) computed from the $100$ realizations of $\hat\theta^\prime(x)$. The computation of $C^{\text{var\,}\hat\phi}_{\ell}$ took $0.098$ seconds. The fast approximation to  $\ell^2 C^{\text{bias\,}\hat\phi}_{\ell}$ is plotted with the \textbf{dashed blue line} in the \textbf{bottom right plot}. This approximation can be seen to be very accurate, nearly indistinguishable from the \textbf{red line}, and took only $0.308$ seconds to compute (compared to $99.44$ seconds for computing  $C^{\text{bias\,}\hat\phi}_{\ell}$ directly).

In Figure \ref{Figure 2 zx and var, extra bias} we show another simulation which is similar the one shown in Figure \ref{Figure 2 zx and var} with the exception that $t_0$ is reduced from $1.5$ to $1.5/7$. This has the effect of shrinking the two-sided cut locus (discussed in Section \ref{section: modeling xi, C and eta}). This is  equivalent to scaling $\phi(x)$ by a factor of $7$ which has the effect of dramatically increasing the bias in the quadratic estimate. Indeed, the main point of Figure \ref{Figure 2 zx and var, extra bias} is to illustrate the qualitative features of the quadratic estimate bias when $\phi(x)$ is too large for the linear approximation in (\ref{first order term, intro}) to hold.
Indeed, when the magnitude of the true $\theta^\prime(x)$ exceeds the two-sided cut locus, the estimate $\hat\theta^\prime(x)$ transitions from a low bias estimate to a bias dominated one as seen in the \textbf{middle plot}. This is presumably due to the ability of the two-sided cut locus to identify when the geodesic path of local spectral densities begins to curl in on itself, creating an ill-posed inversion from observed local spectral density to estimate $\phi(x)$.

\begin{figure}
\includegraphics[height=15cm]{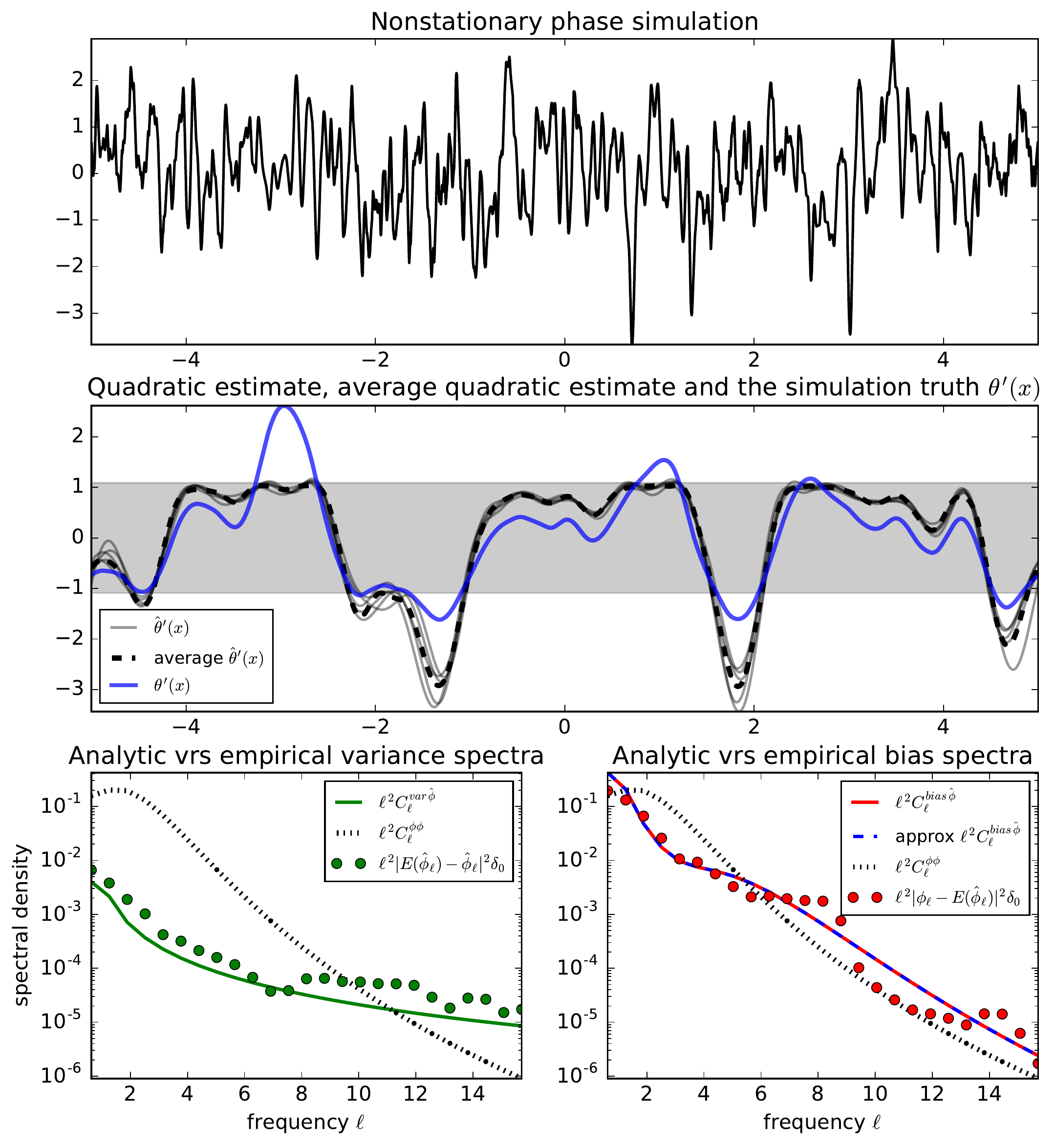}%
\caption{ An illustration of the quadratic estimate bias which results when $\theta^\prime(x)$ is large enough to exceed the \textbf{symmetric two-sided cut locus} discussed Section \ref{section: modeling xi, C and eta} (c.f. Claim \ref{claim: cut locus}). The \textbf{shaded region} shown in the \textbf{middle plot} corresponds to the interior of the symmetric two-sided cut locus. When the true  $\theta^\prime(x)$ (shown in \textbf{blue}) exits the symmetric two-sided cut locus, the quadratic estimate suffers from large bias, attenuating for negative $\theta^\prime(x)$ and amplifying for positive $\theta^\prime(x)$. Note that the only difference between this figure and Figure \ref{Figure 2 zx and var} is the parameter $t_0$ (c.f. Section \ref{section: modeling xi, C and eta}) which was reduced by a factor of $1/7$. This has the effect of shrinking the symmetric two-sided cut locus. All other parameters, including the random seed, are the same.}
\label{Figure 2 zx and var, extra bias}
\end{figure}

\subsection{Nonstationary phase example $d=2$}
\label{section: nonstat example d=2}

In this section we perform a simulation example to illustrate the quadratic estimate applied to a nonstationary spectral phase random field $Z(\bs x)$ in dimension $d=2$. Besides the increase of dimension, there are two main differences in this simulation as compared to the simulation given in Section \ref{section: nonstat example d=1}. The first difference is that the spectral multiplier $\bs \xi_{\bs k}$ is set to $(i\bs k_2, -i\bs k_1)^T$ where $\bs k = (\bs k_1, \bs k_2)$. Therefore $\bs \theta(\bs x) = (\partial_{\bs x_2}\phi(\bs x), -\partial_{\bs x_1}\phi(\bs x))^T$ is a divergent free vector field. The second main difference is that the two spectral densities $C_{\bs k}$ and $\tilde C_{\bs k}$, defined by (\ref{eq: Ck matern}) and (\ref{eq: tildeCk matern}), have different Mat\'ern parameter values as those used in Section \ref{section: nonstat example d=1}. The parameter values for $C_{\bs k}$  are given by $(\nu, \rho, \sigma^2):=(1.5, 0.015, 1)$ and the parameter values for $\tilde C_{\bs k}$ are given by $(\tilde \nu, \tilde\rho, \sigma^2):=(1.7, 0.014, 1)$. Recall that $C_{\bs k}$ and $\tilde C_{\bs k}$ are used to generate $\bs \eta_{\bs k}$ (c.f. Section \ref{section: modeling xi, C and eta}) by requiring both $C_{\bs k}$ and $\tilde C_{\bs k}$ be locally attainable spectral models in $Z(\bs x)$. Therefore the corresponding quadratic estimate is tuned to detect not only a variation in the smoothness of $Z(\bs x)$ but also a corresponding scale change, where the correspondence is related inversely (an increase in smoothness corresponding to a reduction of spatial scale and vice versa).

For this simulation example, the process $Z(\bs x)$  is defined on $[-\pi,\pi)^2$ with periodic boundary conditions and the observed process $Z^{obs}(\bs x)$ is generated without additive noise on a evenly spaced grid of size $400 \times 400$. Just as in Section \ref{section: nonstat example d=1}, a simple discrete Riemann sum approximation, at each observed $\bs x$, was used to approximate to the integral (\ref{eq: nonstate phase}) for generating the simulation of $Z^{obs}(\bs x)$.
The ground truth potential $\phi(\bs x)$, used to generate the nonstationarity in $Z(\bs x)$, is simulated from a mean zero stationary Gaussian process with Mat\'ern parameters $(\nu, \rho, \sigma^2):=(5, 0.3 \pi, 0.5^2)$. Just as in Section \ref{section: nonstat example d=1} the quantities $C^{(0)}_{\bs \ell}, C^{(1)}_{\bs \ell}, C^{(2)}_{\bs \ell}$ and $C^{ZZobs}_{\bs \ell}$ used to generate  $\hat\phi_{\bs\ell}$, $C^{\text{var\,}\hat\phi}_{\bs\ell}$ and $C^{\text{bias\,}\hat\phi}_{\bs \ell}$ are determined by (\ref{eq: C(1), C(2) and C(3) for nonstationary phase}) and (\ref{eq: CZZobs for nonstationary phase}). Finally, the parameter $t_0$ used in (\ref{eq: modeling eta}) to determine $\bs \eta_k$ is set to $1.5$.

Figure \ref{Figure 3} graphically summarizes the simulation results. The \textbf{top left image} shows the quadratic estimate $\hat\phi(\bs x)$ and the \textbf{top right image} shows the ground truth $\phi(\bs x)$. These top images are intended to illustrate the high accuracy of the estimate. The \textbf{bottom right image} shows the data $Z^{obs}(\bs x)$ used in the estimate $\hat\phi(\bs x)$. The \textbf{bottom left plot} shows the radial profile of $|\bs \ell|^2C_{\bs \ell}^{\text{var }\hat\phi}$ (\textbf{solid green}), $|\bs \ell|^2C_{\bs \ell}^{\phi\phi}$ (\textbf{dotted black}) along with the fast approximation to $|\bs \ell|^2C_{\bs \ell}^{\text{bias }\hat\phi}$ (\textbf{dashed blue}) and the corresponding radially averaged empirical mean squared error per wavenumber (\textbf{dotted green}). The computation of $\hat\phi_{\bs \ell}$ and  $C_{\bs \ell}^{\text{var }\hat\phi}$ took $0.25$ seconds and $0.23$ seconds to compute, respectively. The fast approximation to $C_{\bs \ell}^{\text{bias }\hat\phi}$ took $77$ seconds (the exact value of $C_{\bs \ell}^{\text{bias }\hat\phi}$ is not computed in this case since the imputation is intensive and takes on the order of hours in our implementation and is not shown).

%
%
\section{Discusssion}
\label{section: Discussion}

Part of the motivation for this paper is an attempt to construct an extended class of nonstationary random fields, and a corresponding generalized quadratic estimate, which share the same attractive statistical properties of an estimate originally developed for gravitational lensing studies of the Cosmic Microwave Background \cite{hu2001mapping, hu2002mass}.  In doing so we have identified a particular form of nonstationarity, we call {\it local invariance}, which encourages a delicate cancellation of estimation bias. This local invariant property---we believe---is the main source of what makes the gravitational lensing estimates so successful.  Indeed, the generalized quadratic estimate, derived in Section~\ref{section: local invariant}, shares many of the same attractive statistical features as the original gravitational lensing estimate: it is particularly adept at detecting small departures from stationarity and allows fast, accurate quantification of mean square sampling properties.  In Section~\ref{section: NPhase} we focus on a particular subclass of locally invariant nonstationary random fields which are given by a spatially varying spectral phase modulation of a stationary random field. In this work, the theory of optimal transport and the $L_2$-Wasserstein metric play a major role in characterizing the behavior of the set of possible local spectral densities under these models and leads to a natural heuristic for quantifying estimation bias in terms of the Wasserstein geodesic cut locus (see Claim \ref{claim: cut locus}, Section \ref{section: modeling xi, C and eta} and Figure \ref{Figure 2 zx and var, extra bias}).

One of the byproducts of this paper is the understanding that a nonstationary spectral phase can be estimated by analyzing the correlation among the Fourier coefficients of the nonstationary random field $Z(\bs x)$. This was illustrated in Section~\ref{section: NPhase} using a quadratic estimate to reconstruct a spatially varying spectral phase modulation of a stationary random field. Left unanswered, however, is the question of how one simultaneously estimates both the phase and the magnitude of the spectral modulation $A(\bs k, \bs x)$ in model (\ref{eq: intro nonstat phase model}). It appears this line of research has the potential to merge the seminal work of Dahlhaus \cite{dahlhaus1997fitting,dahlhaus2000likelihood} with the generalized quadratic estimate, presented here, for nonstationary estimation within a broad class of nonstationary random fields.

It is also important to mention the fact that we have derived our results under the rather idealized assumption that the observations locations form a dense regular grid and $Z(\bs x)$ has periodic boundary conditions. Extensions to more realistic experimental conditions are not in the scope of this paper but are clearly important for real life applications. The situation is not hopeless, however, since these same features are ubiquitous in measurements of the Cosmic Microwave Background. Despite this, Cosmologists have devised methods which turn the idealized quadratic estimate into a pragmatic statistical tool for probing gravitational lensing (see \cite{namikawa2013bias,van2012measurement,planck2013lensing,planck2015lensing}, for example). This suggests there exist analogous methods which can make the generalized quadratic estimate available to more general observational scenarios.

We finish with a discussion of Assumption \ref{Assumption 2} that stipulates $\bs \theta(\bs x)$ be characterized by a scalar potential $\phi(\bs x)$.  It is yet unclear how one generalizes this assumption, especially in the case where $\bs \theta(\bs x)$ maps into a higher dimensional space $\Bbb R^m$ for $m>d$. Notice that by considering a general $\bs \theta(\bs x)\colon\Bbb R^d \rightarrow \Bbb R^m$ one may redefine $\bs\theta(\bs x)$ by absorbing (i.e. concatenating) the observation locations $\bs x$ into extra coordinates of $\bs \theta(\bs x)$. This generalization enables one to replace the local invariant condition $\text{cov}(Z(\bs x), Z(\bs y))=K(\bs x-\bs y, \bs\theta(\bs x) - \bs \theta(y))$ with the more general condition
\begin{align}
\label{eq: embedding}
\text{cov}(Z(\bs x), Z(\bs y))=K(\bs\theta(\bs x) - \bs \theta(\bs y)).
\end{align}
Random fields $Z(\bs x)$, which have a covariance function of the form (\ref{eq: embedding}), are simply traces of stationary random fields defined on the higher dimensional space $\Bbb R^m$, restricted to the $d$-dimensional parameterized surface $\{\bs\theta(\bs x):\bs x\in\Bbb R^d\}$. Viewed from this perspective, it appears plausible that there exists a deeper, more geometric, picture of local invariance and quadratic estimation. It is not yet clear whether or not this viewpoint is useful, but it is tempting to imagine that the generalized quadratic estimate is simply a manifold embedding estimate in disguise. If such a development bears theoretical fruit, it would be a major step in the direction of a unified statistical theory of nonstationary random fields.

\begin{figure}
\includegraphics[height=15cm]{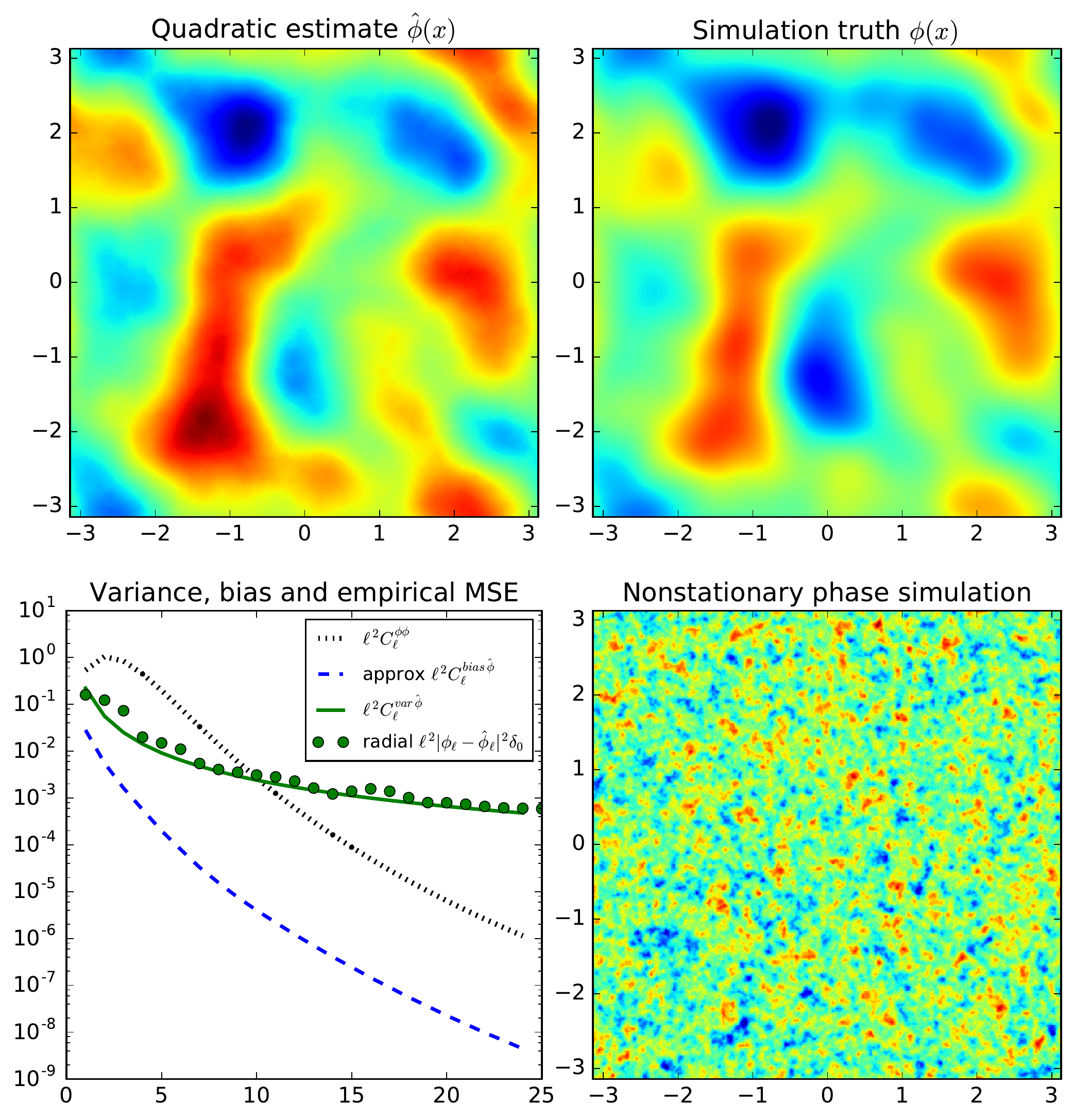}%
\caption{An illustration of the quadratic estimate in $d=2$ for a nonstationary spectral phase random field $Z(\bs x)$. The model for $Z(\bs x)$ is defined to have local variation in both the local smoothness of $Z(\bs x)$ and a local range parameter, where an increase of local smoothness corresponds reduction of local range and vice versa. The \textbf{top left image} shows the quadratic estimate $\hat\phi(\bs x)$ with the  ground truth $\phi(\bs x)$ shown in the \textbf{top right image} and the data shown in the \textbf{bottom right image}.  The \textbf{bottom left plot} shows the radial profile of  $|\bs \ell|^2C_{\bs \ell}^{\text{var }\hat\phi}$ (\textbf{solid green}), $|\bs \ell|^2C_{\bs \ell}^{\phi\phi}$ (\textbf{dotted black}) along with the fast approximation to $|\bs \ell|^2C_{\bs \ell}^{\text{bias }\hat\phi}$ (\textbf{dashed blue}) and the corresponding radially averaged empirical mean squared error per wavenumber (\textbf{dotted green}). See Section \ref{section: nonstat example d=2} for further simulation details.
 }
\label{Figure 3}
\end{figure}

%
%

\bibliography{paper}

\begin{thebibliography}{10}

\bibitem{ambrosio2008gradient}
L.~Ambrosio, N.~Gigli, and G.~Savar{\'e}.
\newblock {\em Gradient flows: in metric spaces and in the space of probability
  measures}.
\newblock Springer Science \& Business Media, 2008.

\bibitem{dahlhaus1997fitting}
R.~Dahlhaus.
\newblock Fitting time series models to nonstationary processes.
\newblock {\em Annals of Statistics}, 25(1):1--37, 1997.

\bibitem{dahlhaus2000likelihood}
R.~Dahlhaus.
\newblock A likelihood approximation for locally stationary processes.
\newblock {\em Annals of Statistics}, pages 1762--1794, 2000.

\bibitem{das2011detection}
S.~Das et~al.
\newblock Detection of the power spectrum of cosmic microwave background
  lensing by the {A}tacama cosmology telescope.
\newblock {\em Physical Review Letters}, 107(2):021301, 2011.

\bibitem{fuglstad2015exploring}
G.~Fuglstad, F.~Lindgren, D.~Simpson, and H.~Rue.
\newblock Exploring a new class of non-stationary spatial gaussian random
  fields with varying local anisotropy.
\newblock {\em Statistica Sinica}, 25(1):115--133, 2015.

\bibitem{fuglstad2015does}
G.~Fuglstad, D.~Simpson, F.~Lindgren, and H.~Rue.
\newblock Does non-stationary spatial data always require non-stationary random
  fields?
\newblock {\em Spatial Statistics}, 14:505--531, 2015.

\bibitem{Gikhman_Skorokhod_book_v1}
I.~Gikhman and A.~Skorokhod.
\newblock {\em The Theory of Stochastic Processes I}.
\newblock Classics in Mathematics. Springer Berlin Heidelberg, 2015.

\bibitem{giner2014monotonically}
G.~Giner and G.~Smyth.
\newblock A monotonically convergent {N}ewton iteration for the quantiles of
  any unimodal distribution, with application to the inverse {G}aussian
  distribution.
\newblock 2014.

\bibitem{hsing2016local}
T.~Hsing, T.~Brown, and B.~Thelen.
\newblock Local intrinsic stationarity and its inference.
\newblock {\em Annals of Statistics}, To Appear.

\bibitem{hu2001mapping}
W.~Hu.
\newblock Mapping the dark matter through the cosmic microwave background
  damping tail.
\newblock {\em The Astrophysical Journal Letters}, 557(2):L79, 2001.

\bibitem{hu2002mass}
W.~Hu and T.~Okamoto.
\newblock Mass reconstruction with cosmic microwave background polarization.
\newblock {\em The Astrophysical Journal}, 574(2):566, 2002.

\bibitem{ibragimov2012gaussian}
I.~Ibragimov and Yurii~A. Rozanov.
\newblock {\em Gaussian random processes}, volume~9.
\newblock Springer Science \& Business Media, 2012.

\bibitem{isserlis1916certain}
L.~Isserlis.
\newblock On certain probable errors and correlation coefficients of multiple
  frequency distributions with skew regression.
\newblock {\em Biometrika}, 11(3):185--190, 1916.

\bibitem{namikawa2013bias}
T.~Namikawa, D.~Hanson, and R.~Takahashi.
\newblock Bias-hardened {CMB} lensing.
\newblock {\em Monthly Notices of the Royal Astronomical Society},
  431(1):609--620, 2013.

\bibitem{paciorek2006spatial}
C.~Paciorek and M.~Schervish.
\newblock Spatial modelling using a new class of nonstationary covariance
  functions.
\newblock {\em Environmetrics}, 17(5):483--506, 2006.

\bibitem{planck2013lensing}
{Planck Collaboration}.
\newblock {Planck 2013 results. XVII. Gravitational lensing by large-scale
  structure}.
\newblock {\em Astronomy and Astrophysics}, 571:A17, November 2014.

\bibitem{planck2015lensing}
{Planck Collaboration}.
\newblock {Planck 2015 results. XV. Gravitational lensing}.
\newblock {\em ArXiv e-prints}, February 2015.

\bibitem{priestley1965evolutionary}
M.~Priestley.
\newblock Evolutionary spectra and non-stationary processes.
\newblock {\em Journal of the Royal Statistical Society. Series B
  (Methodological)}, pages 204--237, 1965.

\bibitem{priestley1981spectral}
M.~Priestley.
\newblock {\em Spectral analysis and time series}, volume~2.
\newblock Academic press, 1981.

\bibitem{sampson2010constructions}
P.~Sampson.
\newblock Constructions for nonstationary spatial processes.
\newblock {\em Handbook of Spatial Statistics}, pages 119--130, 2010.

\bibitem{sampson1992nonparametric}
P.~Sampson and P.~Guttorp.
\newblock Nonparametric estimation of nonstationary spatial covariance
  structure.
\newblock {\em Journal of the American Statistical Association},
  87(417):108--119, 1992.

\bibitem{stein2012interpolation}
M.~Stein.
\newblock {\em Interpolation of spatial data: some theory for {K}riging}.
\newblock Springer Science \& Business Media, 2012.

\bibitem{Polarbear2014}
{The Polarbear Collaboration}.
\newblock {A Measurement of the Cosmic Microwave Background B-mode Polarization
  Power Spectrum at Sub-degree Scales with POLARBEAR}.
\newblock {\em The Astrophysical Journal}, 794:171, October 2014.

\bibitem{van2012measurement}
A.~Van~Engelen et~al.
\newblock A measurement of gravitational lensing of the microwave background
  using south pole telescope data.
\newblock {\em The Astrophysical Journal}, 756(2):142, 2012.

\bibitem{villani2003topics}
C.~Villani.
\newblock {\em Topics in optimal transportation}.
\newblock Number~58. American Mathematical Soc., 2003.

\bibitem{wick1950evaluation}
G.~Wick.
\newblock The evaluation of the collision matrix.
\newblock {\em Physical review}, 80(2):268, 1950.

\end{thebibliography}

%
%

\appendix

%
%
\section{Detailed derivations}\label{section: Detailed Proofs}

%
%

\begin{claim}
\label{claim: first order expansion}
Let $\bs\theta(\bs x)\colon \Bbb R^d \rightarrow \Bbb R^d$ be a vector field and $Z(\bs x)$ be a random field which satisfies
$E\big(Z(\bs x)Z(\bs y)|\bs \theta(\cdot)\big)= C^{(0)}(\bs x-
\bs y) + \bs C^{(1)}(\bs x-\bs y)\cdot(\bs\theta(\bs x)-\bs\theta(\bs y)) +  \mathcal O(\bs\theta^2)$. Then
\begin{align}
\label{eq in claim: first order expansion}
E(Z_{\bs k+\bs \ell}Z_{-\bs k}|\bs \theta(\cdot)) &= \bs\theta_{\bs \ell}\cdot\big(\bs C^{(1)}_{\bs k} - \bs C_{\bs k+\bs \ell}^{(1)}\big) + \mathcal O(\bs\theta^2)
\end{align}
when $\bs\ell\neq 0$.
\end{claim}

\begin{proof}
The Fourier transform $C^{(0)}(\bs x-\bs y)$, with respect to $\bs x$ and $\bs y$, gives  $\delta_{\bs \ell_1 + \bs \ell_2} {(2\pi)}^{d/2} C^{(0)}_{\bs \ell_1}$.
Similarly, the Fourier transform of $\bs\theta(\bs x)\cdot \bs C^{(1)}(\bs x-\bs y)$ and $- \bs\theta(\bs y)\cdot \bs C^{(1)}(\bs x-\bs y)$, with respect to $\bs x$ and $\bs y$, gives $\bs\theta_{\bs\ell_1+\bs\ell_2}\cdot  \bs C^{(1)}_{-\bs\ell_2}$ and $-\bs\theta_{\bs \ell_1+\bs \ell_2} \cdot  \bs C^{(1)}_{\bs\ell_1}$ respectively.
Summing these three terms gives
\begin{align*}
E\big(Z_{\bs \ell_1}Z_{\bs \ell_2}|\bs \theta(\cdot)\big) &= {(2\pi)}^{d/2}C^{(0)}_{\bs \ell_1} \delta_{\bs \ell_1+\bs \ell_2}
+ \bs\theta_{\bs \ell_1+\bs \ell_2}\cdot\big(\bs C^{(1)}_{-\bs \ell_2} -\bs C_{\bs \ell_1}^{(1)}\big) + \mathcal O(\bs\theta^2).
\end{align*}
Replacing $\bs \ell_1$ with  $\bs k+\bs \ell$  and $\bs \ell_2$ with $-\bs k$ finishes the derivation.
\end{proof}

%
%

\begin{claim}\label{appendix, claim: quad est derivations}
Suppose assumptions \ref{Assumption 4}, \ref{Assumption 2} and \ref{Assumption 3} hold. Then the first order unbiased quadratic estimate of $\phi_{\bs \ell}$, which corresponds to an approximate inverse variance weighted averaging of $Z_{\bs k+\bs \ell}^{obs}Z_{-\bs k}^{obs}$, has the form
    \begin{align}
        \label{appendix-claim def of quad est}
        \hat\phi_{\bs \ell} &=   A_{\bs \ell}\sum_{p=1}^d \bs\xi^*_{p,\bs\ell} \int
        e^{-i\bs x\cdot\bs\ell}
        \mathscr A(\bs x)\mathscr B_{p}(\bs x) \frac{d\bs x}{{(2\pi)}^{d/2}}
    \end{align}
    where $\mathscr A_{\bs \ell}:= Z^{obs}_{\bs \ell}/ C^{ZZobs}_{\bs \ell}$,  $\mathscr B_{p,\bs\ell} := i 2 \textit{\,imag}(\bs C^{(1)}_{p,\bs \ell}) Z^{obs}_{\bs\ell} / C^{ZZobs}_{\bs \ell}$ and
    the normalizing constant
    $A_{\bs\ell}$ is given by
    \begin{align}
        \label{appendix-claim def of Aℓinv}
        A_{\bs\ell}^{-1} &= \sum_{p,q = 1}^d\bs\xi_{p,\bs \ell}\bs\xi^*_{q,\bs \ell}\int  e^{-i\bs x\cdot\bs\ell} \left[2\mathscr A_{p,q}(\bs x)\mathscr B(\bs x)  - \mathscr C_{p}(\bs x)\mathscr C_{q}(\bs x) -  \mathscr D_{p}(\bs x)\mathscr D_{q}(\bs x)\right]  \frac{d\bs x}{{(2\pi)}^{d/2}}
    \end{align}
    where $\mathscr A_{p,q, \bs \ell}:= \bs C^{(1)}_{p,\bs\ell} \bs C^{{(1)}^*}_{q,\bs\ell}  / C^{ZZobs}_{\bs\ell}$, $\mathscr B_{\bs\ell}:= 1/C^{ZZobs}_{\bs\ell}$, $\mathscr C_{p,\bs\ell}:= \bs C^{(1)}_{p,\bs\ell}  / C^{ZZobs}_{\bs\ell}$ and $\mathscr D_{p,\bs\ell}:= \bs C^{{(1)}^*}_{p,\bs\ell}  / C^{ZZobs}_{\bs\ell}$.
\end{claim}

\begin{proof}
By Claim \ref{claim: first order expansion} we have that
\begin{align}
\label{eq: first order appendix behavior of ZZ}
E\big(Z^{obs}_{\bs k+\bs \ell}Z^{obs}_{-\bs k}\big)
= \phi_{\bs \ell}\, (\bs \xi_{\bs \ell}\cdot \bs f_{\bs k,\bs\ell})+ \mathcal O(\phi^2)
\end{align}
 when $\bs \ell\neq 0$ where $\bs f_{\bs k,\bs \ell}:= \bs C_{\bs k}^{(1)} -  \bs C_{\bs k+\bs \ell}^{(1)}$.
 Therefore the quadratic estimate, as a weighted average of the first order unbiased terms $Z^{obs}_{\bs k+\bs \ell}Z^{obs}_{-\bs k}/(\bs \xi_{\bs \ell}\cdot \bs f_{\bs k,\bs\ell})$, can be written in the form
\begin{align}
    \label{the basic form of the quadratic estimator in the proof}
\hat\phi_{\bs\ell}
& = \int w_{\bs k,\bs \ell}\frac{Z^{obs}_{\bs k+\bs \ell}Z^{obs}_{-\bs k}}{\bs \xi_{\bs \ell}\cdot \bs f_{\bs k,\bs\ell}} \frac{d\bs k}{{(2\pi)}^{d/2}}
\end{align}
where $w_{\bs k,\bs\ell}\geq 0$ are normalized so that $\hat\phi_{\bs\ell}$ has expected value $\phi_{\bs\ell} + \mathcal O(\phi^2)$ using (\ref{eq: first order appendix behavior of ZZ}). Assuming $\bs\ell\neq 0$, the Gaussian part of the variance of $Z^{obs}_{\bs k+\bs \ell}Z^{obs}_{-\bs k}/(\bs \xi_{\bs \ell}\cdot \bs f_{\bs k,\bs\ell})$ can be computed as follows
\begin{align*}
\text{var}\left(\frac{Z^{obs}_{\bs k+\bs \ell}Z^{obs}_{-\bs k}}{\bs \xi_{\bs \ell}\cdot \bs f_{\bs k,\bs\ell}}
\right)
& = \frac{1}{|\bs \xi_{\bs \ell}\cdot \bs f_{\bs k,\bs\ell}|^2}\Big[E(Z^{obs}_{\bs k+\bs\ell}Z^{obs}_{-\bs k}Z^{obs}_{-\bs k-\bs\ell}Z^{obs}_{\bs k})  - E(Z^{obs}_{\bs k+\bs\ell}Z^{obs}_{-\bs k})E(Z^{obs}_{-\bs k-\bs\ell}Z^{obs}_{\bs k}) \Big] \\
&\approx \frac{1}{|\bs \xi_{\bs \ell}\cdot \bs f_{\bs k,\bs\ell}|^2}\underbrace{\Bigl[E(Z^{obs}_{\bs k+\bs\ell}Z^{obs}_{-\bs k-\bs\ell})E(Z^{obs}_{\bs k}Z^{obs}_{-\bs k})  + E(Z^{obs}_{\bs k+\bs\ell}Z^{obs}_{\bs k})E(Z^{obs}_{-\bs k}Z^{obs}_{-\bs k-\bs\ell}) \Bigr]}_{\textit{only keeping the Gaussian part of the trispectrum}} \\
& = \frac{C^{ZZobs}_{\bs k+\bs\ell} C^{ZZobs}_{\bs k}}{|\bs \xi_{\bs \ell}\cdot \bs f_{\bs k,\bs\ell}|^2}\big[\delta_{\bs 0}^2   + \delta_{2\bs k+\bs\ell}^2 \big].
\end{align*}
where $C^{ZZobs}_{\bs\ell}$ denotes the spectral density of $Z^{obs}(\bs x)$ marginalized over $\phi(\cdot)$.
If we ignore the term $\delta_{2\bs k+\bs\ell}^2$, which only activates at the point $\bs k = -\bs\ell/2$, then by defining $w_{\bs k,\bs \ell}$ in (\ref{the basic form of the quadratic estimator in the proof}) to be proportional to the approximate inverse (Gaussian part of the) variance of $Z^{obs}_{\bs k+\bs \ell}Z^{obs}_{-\bs k}/(\bs \xi_{\bs \ell}\cdot \bs f_{\bs k,\bs\ell})$ one has
\begin{align}
\hat\phi_{\bs\ell}
& = A_{\bs\ell}\int \frac{|\bs \xi_{\bs \ell}\cdot \bs f_{\bs k,\bs\ell}|^2}{C^{ZZobs}_{\bs k+\bs \ell} C^{ZZobs}_{\bs k}}\frac{Z^{obs}_{\bs k+\bs \ell}Z^{obs}_{-\bs k}}{\bs \xi_{\bs \ell}\cdot \bs f_{\bs k,\bs\ell}} \frac{d\bs k}{{(2\pi)}^{d/2}} \nonumber\\
& =  A_{\bs\ell}\sum_{p=1}^d \bs\xi^*_{p,\bs \ell}\int
\Bigl[
  \frac{Z^{obs}_{\bs k+\bs \ell}}{C^{ZZobs}_{\bs k+\bs \ell} }\frac{\bs C^{(1)}_{p,-\bs k} Z^{obs}_{-\bs k}}{C^{ZZobs}_{-\bs k}}
- \frac{\bs C^{{(1)}^*}_{p, \bs k+\bs\ell}Z^{obs}_{\bs k+\bs \ell}}{C^{ZZobs}_{\bs k+\bs \ell} } \frac{Z^{obs}_{-\bs k}}{C^{ZZobs}_{-\bs k}}
\Bigr] \frac{d\bs k}{{(2\pi)}^{d/2}} \nonumber\\
&= A_{\bs\ell}\sum_{p=1}^d \bs\xi^*_{p,\bs \ell} \int
\left[
\mathscr A_{{\bs k+\bs \ell}}\mathscr D_{p,-\bs k}  -
\mathscr C_{p,\bs k+\bs \ell}\mathscr A_{-\bs k}
 \right]  \frac{d\bs k}{{(2\pi)}^{d/2}} \nonumber \\
 &= A_{\bs\ell}\sum_{p=1}^d \bs\xi^*_{p,\bs \ell} \int
 e^{-i\bs x\cdot \bs\ell}
 \left[
 \mathscr A(\bs x)\mathscr D_{p}(\bs x)  -
 \mathscr C_{p}(\bs x)\mathscr A(\bs x)
  \right]  \frac{d\bs x}{{(2\pi)}^{d/2}} \label{first quad est form in proof}
\end{align}
where $\mathscr A_{\bs\ell}:= Z^{obs}_{\bs\ell}/ C^{ZZobs}_{\bs\ell}$, $\mathscr D_{p,\bs\ell}:=\bs C^{(1)}_{p,\bs\ell}Z^{obs}_{\bs\ell}/C^{ZZobs}_{\bs\ell}$ and $\mathscr C_{p,\bs\ell}:=\bs C^{{(1)}^*}_{p,\bs\ell}Z^{obs}_{\bs\ell}/C^{ZZobs}_{\bs\ell}$. Notice that the Fourier transform of $\mathscr D_{p}(\bs x)  - \mathscr C_{p}(\bs x)$ can be simplified as follows
\begin{align*}
\mathscr D_{p,\bs \ell}  - \mathscr C_{p,\bs \ell}
& = \bigl[\bs C^{(1)}_{p,\bs \ell} - \bs C^{{(1)}^*}_{p,\bs \ell}\bigr]\frac{Z^{obs}_{\bs\ell}}{C^{ZZobs}_{\bs\ell}}
 = i 2 \textit{\,imag}(\bs C^{(1)}_{p,\bs\ell})\frac{Z^{obs}_{\bs\ell}}{C^{ZZobs}_{\bs\ell}}.
\end{align*}
This gives (\ref{appendix-claim def of quad est}) as was to be shown.

The normalizing constant $A_{\bs\ell}$ is defined so that the right hand of (\ref{first quad est form in proof}) is unbiased (up to first order). Utilizing (\ref{eq: first order appendix behavior of ZZ}) this unbiased constraint is written as follows
\begin{align}
 1 &= A_{\bs\ell}\int \frac{|\bs \xi_{\bs \ell}\cdot \bs f_{\bs k,\bs\ell}|^2}{C^{ZZobs}_{\bs k+\bs \ell} C^{ZZobs}_{\bs k}} \frac{d\bs k}{{(2\pi)}^{d/2}}  \nonumber\\
&= A_{\bs\ell}\int \frac{\bigl|\bs \xi_{\bs \ell}\cdot\bs C^{(1)}_{\bs k+\bs \ell}  - \bs\xi_{\bs \ell}\cdot\bs C^{(1)}_{\bs k} \bigr|^2}{C^{ZZobs}_{\bs k+\bs \ell} C^{ZZobs}_{\bs k}} \frac{d\bs k}{{(2\pi)}^{d/2}}
\nonumber\\
&= A_{\bs\ell}\sum_{p,q = 1}^d\bs\xi_{p,\bs\ell}\bs\xi^*_{q,\bs\ell}\int \frac{\bs C^{(1)}_{p,\bs k+\bs \ell} \bs C^{{(1)}^*}_{q,\bs k+\bs \ell} + \bs C^{(1)}_{p,\bs k} \bs C^{{(1)}^*}_{q,\bs k}  -     \bs C^{(1)}_{p,\bs k+\bs \ell} \bs C^{{(1)}^*}_{q,\bs k} - \bs C^{{(1)}^*}_{p,\bs k+\bs \ell} \bs C^{(1)}_{q,\bs k} }{C^{ZZobs}_{\bs k+\bs\ell} C^{ZZobs}_{\bs k}} \frac{d\bs k}{{(2\pi)}^{d/2}}  \nonumber\\
&= A_{\bs\ell}\sum_{p,q = 1}^d\bs\xi_{p,\bs\ell}\bs\xi^*_{q,\bs\ell}\int  e^{-i\bs x\cdot\bs \ell} \left[2\mathscr A_{p,q}(\bs x)\mathscr B(\bs x)  - \mathscr C_{p}(\bs x)\mathscr C_{q}(\bs x) -  \mathscr D_{p}(\bs x)\mathscr D_{q}(\bs x)\right]  \frac{d\bs x}{{(2\pi)}^{d/2}}\label{Aell derivation 11}
\end{align}
where $\mathscr A_{p,q, \bs\ell}:= \bs C^{(1)}_{p,\bs\ell} \bs C^{{(1)}^*}_{q,\bs\ell}  / C^{ZZobs}_{\bs\ell}$, $\mathscr B_{\bs\ell}:= 1/C^{ZZobs}_{\bs\ell}$, $\mathscr C_{p,\bs\ell}:= \bs C^{(1)}_{p,\bs\ell}  / C^{ZZobs}_{\bs\ell}$ and $\mathscr D_{p,\bs\ell}:= \bs C^{{(1)}^*}_{p,\bs\ell}  / C^{ZZobs}_{\bs\ell}$.

\end{proof}

%
%

\begin{claim}[{\bf Estimation variance}]\label{thm: quantify the var fluctuations of hat phi}
    Suppose $X(\bs x)$ is a mean-zero Gaussian random field with spectral density given by $C^{X\!X}_{\bs \ell}$.
    Then the spectral density of $\hat\phi_{\bs \ell}\{X,\!X\}$ (c.f. Definition \ref{def: quad est applied to X}), which satisfies $\delta^{\phantom{*}}_{\bs \ell-\bs \ell^\prime}C^{\text{\rm var }\hat\phi}_{\bs \ell} = E\big(\hat\phi_{\bs \ell}\{X,\!X\} \, \hat\phi_{\bs \ell^\prime}\{X,\!X\}^*\big) $, is given as follows
    \begin{align}
        C_{\bs\ell}^{\text{\rm var }\hat{\phi}}
        &= 2 A_{\bs\ell}^2
        \int
        \Bigl|\bs\xi_{\bs\ell}\!\cdot\!\bs C^{{(1)}}_{\bs k} - \bs\xi_{\bs\ell}\!\cdot\! \bs C^{{(1)}}_{\bs k+\bs \ell}\Bigr|^2
        \frac{C^{X\!X}_{\bs k+\bs\ell}}{(C^{ZZobs}_{\bs k+\bs\ell})^2}
        \frac{C^{X\!X}_{\bs k}}{(C^{ZZobs}_{\bs k})^2}
         \frac{d\bs k}{{(2\pi)}^{d}} \label{appendix, claim raw var} \\
        &= 2 A_{\bs\ell}^2  \sum_{p,q = 1}^d\bs\xi_{p,\bs\ell}\bs\xi^*_{q,\bs\ell}\int e^{-i\bs x\cdot\bs \ell} \left[2\mathscr A_{p,q}(\bs x)\mathscr B(\bs x)  - \mathscr C_{p}(\bs x)\mathscr C_{q}(\bs x) -  \mathscr D_{p}(\bs x)\mathscr D_{q}(\bs x)\right]  \frac{d\bs x}{{(2\pi)}^{d}} \label{appendix, claim fast var}
    \end{align}
    for all $\bs \ell\neq 0$
    where $\mathscr A_{p,q, \bs \ell}:= \bs C^{(1)}_{p,\bs \ell} \bs C^{{(1)}^*}_{q,\bs \ell} C^{X\!X}_{\bs\ell} / (C^{ZZobs}_{\bs \ell})^2$, $\mathscr B_\ell:=  C^{X\!X}_{\bs\ell} / (C^{ZZobs}_{\bs\ell})^2$, $\mathscr C_{p,\bs\ell}:= \bs C^{(1)}_{p,\bs\ell}  C^{X\!X}_{\bs\ell} / (C^{ZZobs}_{\bs\ell})^2$ and $\mathscr D_{p,\bs\ell}:= \bs C^{{(1)}^*}_{p,\bs\ell}  C^{X\!X}_{\bs\ell} / (C^{ZZobs}_{\bs\ell})^2$.
    Moreover, if $C_{\bs\ell}^{X\!X} = C^{ZZobs}_{\bs\ell}$  one obtains
    \begin{equation}
    C_{\bs\ell}^{\text{\rm var }\hat{\phi}} =  2{(2\pi)}^{-d/2} A_{\bs\ell}.
    \end{equation}
\end{claim}

\begin{proof}
    Recall Definition \ref{def: quad est applied to X} which states that $\hat\phi_{\bs \ell}\{X,\!X\}$ denotes the quadratic estimate applied to data $X(\bs x)$. Therefore
    \begin{align}
        E\big(\hat\phi_{\bs \ell}\{X,\!X\} \, \hat\phi_{\bs \ell^\prime}\{X,\!X\}^*\big)
        & = A_{\bs \ell}A_{\bs \ell^\prime}
        \iint
        {\Big(\bs\xi_{\bs \ell} \!\cdot\!\bs C^{{(1)}}_{\bs k} - \bs\xi_{\bs \ell} \!\cdot\!\bs C^{{(1)}}_{\bs k+\bs \ell}\Big)}^{\! *}
        {\Big(\bs\xi_{\bs \ell^\prime} \!\cdot\!\bs C^{{(1)}}_{\bs k^\prime} - \bs\xi_{\bs \ell^\prime} \!\cdot\!\bs C^{{(1)}}_{\bs k^\prime+\bs \ell^\prime}\Big)}^{}\nonumber\\
        &\qquad\qquad\qquad
        \times
        \frac{E\bigl(X_{\bs k+\bs \ell}X_{-\bs k}X_{-\bs k^\prime-\bs\ell^\prime}X_{\bs k^\prime}\bigr)}{C^{ZZobs}_{\bs k+\bs \ell}C^{ZZobs}_{\bs k}C^{ZZobs}_{\bs k^\prime+\bs \ell^\prime}C^{ZZobs}_{\bs k^\prime}}
        \frac{d\bs kd\bs k^\prime}{{(2\pi)}^{d}}.\label{wicks term in Cvar}
    \end{align}
    Now expanding the above fourth moment, using Wick's theorem (also called Isserlis's Theorem) \cite{wick1950evaluation, isserlis1916certain} and the Gaussianity of $X(\bs x)$, one obtains
    \begin{align*}
        E\bigl(X_{\bs k+\bs \ell}X_{-\bs k}X_{-\bs k^\prime-\bs\ell^\prime}X_{\bs k^\prime}\bigr)
        &=C^{X\!X}_{\bs k+\bs \ell}C^{X\!X}_{-\bs k^\prime}\bigl( \delta_{\bs\ell}\delta_{\bs\ell^\prime} +  \delta_{\bs k-\bs k^\prime + \bs \ell-\bs\ell^\prime}\delta_{\bs k-\bs k^\prime} +  \delta_{\bs k + \bs k^\prime +\bs\ell}\delta_{\bs k + \bs k^\prime +\bs\ell^\prime}   \bigr) \\
        &=C^{X\!X}_{\bs k+\bs\ell}C^{X\!X}_{-\bs k^\prime}\bigl( \delta_{\bs\ell}\delta_{\bs\ell^\prime} +  \delta_{
        \bs\ell-\bs\ell^\prime}\delta_{\bs k-\bs k^\prime} +  \delta_{\bs\ell-\bs\ell^\prime}\delta_{\bs k+\bs k^\prime +\bs\ell^\prime}   \bigr).
    \end{align*}
    The term $\delta_{\bs\ell}\delta_{\bs\ell^\prime}$ is only nonzero  when $\bs\ell=\bs\ell^\prime=0$. Furthermore, by a change of variables, one can see that effect of the terms $\delta_{ \bs\ell-\bs\ell^\prime}\delta_{\bs k-\bs k^\prime}$ and $ \delta_{\bs\ell-\bs\ell^\prime}\delta_{\bs k+\bs k^\prime +\bs \ell^\prime}$ in (\ref{wicks term in Cvar}) are identical. Therefore assuming $\bs \ell \neq 0$ and replacing  $E\bigl(X_{\bs k+\bs \ell}X_{-\bs k}X_{-\bs k^\prime-\bs\ell^\prime}X_{\bs k^\prime}\bigr)$ in (\ref{wicks term in Cvar}) with $2C^{X\!X}_{\bs k+\bs\ell}C^{X\!X}_{-\bs k^\prime}\delta_{ \bs\ell-\bs\ell^\prime}\delta_{\bs k-\bs k^\prime}$ gives (\ref{appendix, claim raw var}). A similar approach to the derivation for (\ref{Aell derivation 11}) can be used to establish (\ref{appendix, claim fast var}). Finally, if one replaces $C^{X\!X}_{\bs k}$ in (\ref{appendix, claim raw var}) with the marginal spectral density $C^{ZZobs}_{\bs k}$ then using (\ref{Aell derivation 11}) one obtains $C_{\bs\ell}^{\text{\rm var }\hat{\phi}} =  2{(2\pi)}^{-d/2} A_{\bs\ell}$.

\end{proof}

%
%

\begin{claim}[{\bf Estimation bias}]\label{thm: bias in the appendix}
    Let $\mathcal O(\phi^2)(\bs x,\bs y) := {(\bs\theta(\bs x)-\bs\theta(\bs y))}^T \bs C^{(2)}(\bs x-\bs y) (\bs\theta(\bs x)-\bs\theta(\bs y))$ where $\bs C^{(2)}(\bs x)\colon\Bbb R^d \rightarrow \Bbb R^{d\times d}$ and $\bs C^{(2)}(-\bs x)=\bs C^{(2)}(\bs x)$. Then
    \begin{align}
    	\label{eq: the form of Op2}
        \mathcal O(\phi^2)_{\bs k+\bs\ell, -\bs k}
        &=\sum_{p,q = 1}^d
        \int \bs\theta_{p,\bs\omega}\bs\theta_{q,\bs\ell - \bs\omega}
		 \Bigl(
         {\bs C^{(2)}_{p,q,\bs k}}
        + {\bs C^{(2)}_{p,q,\bs k +\bs \ell}}
        - {\bs C^{(2)}_{p,q,\bs k +\bs \ell - \bs\omega}}
        - {\bs C^{(2)}_{p,q,\bs k + \bs\omega}}
        \Bigr)\frac{d\bs\omega}{(2\pi)^{d/2}}.
    \end{align}
    Define $\hat\phi_{\bs \ell}^{\text{bias}}:=\hat\phi_{\bs \ell}\big\{\mathcal O(\phi^2)_{\bs k+\bs \ell, -\bs k}\big\}$ (see Definition \ref{def: quad est applied to X}), where the quadratic estimate $\hat\phi_{\bs \ell}$ satisfies all the assumptions given in Claim \ref{appendix, claim: quad est derivations}. Then
    \begin{align}
    	\label{eq: the form of hatphibias}
        \hat\phi_{\bs \ell}^{\text{bias}}&= 2\sum_{p,q=1}^d
    	\int
    	\bs\theta_{p,\bs\omega}\bs\theta_{q,\bs \ell - \bs\omega}\, \hat\phi_{\bs \ell}\Big\{{\bs C^{(2)}_{p,q,\bs k}}- {\bs C^{(2)}_{p,q,\bs k + \bs\omega}}\Big\}\frac{d\bs\omega}{(2\pi)^{d/2}}
   	\end{align}
    and the spectral density of $\hat\phi_{\bs \ell}^{\text{bias}}$, under a mean zero Gaussian random field model for $\bs\theta(\bs x)$ with spectral density matrix $C_{\bs \ell}^{\bs \theta\bs\theta}$, satisfies
	\begin{align}
    	\label{eq: the form of the spectrum for hatphibias}
		E\big(\hat\phi_{\bs \ell}^{\text{bias}} \hat\phi_{\bs \ell^\prime}^{\text{bias}^*}\big)
        &=
        4\delta_{\bs\ell-\bs \ell^\prime}\!\!\! \sum_{p,q,p^\prime,q^\prime=1}^d \int
        \Big(
        C^{\bs \theta\bs \theta}_{p,p^\prime,\bs \omega}C^{\bs \theta\bs \theta}_{q,q^\prime,\bs \ell - \bs \omega}
        + C^{\bs \theta\bs \theta}_{p,q^\prime,\bs \omega}C^{\bs \theta\bs \theta}_{q,p^\prime,\bs \ell - \bs \omega}
        \Big)\nonumber
    	\\
        &\qquad\qquad\qquad\qquad\qquad\times
        \hat\phi_{\bs \ell}\Big\{{\bs C^{(2)}_{p,q,\bs k}}- {\bs C^{(2)}_{p,q,\bs k + \bs\omega}}\Big\}
        \hat\phi_{\bs \ell}\Big\{{\bs C^{(2)}_{p^\prime,q^\prime,\bs k}}- {\bs C^{(2)}_{p^\prime,q^\prime,\bs k + \bs\omega}}\Big\}^*\frac{d\bs\omega}{(2\pi)^{d}}
        \end{align}
        when $\bs \ell \neq 0$ or $\bs \ell^\prime \neq 0$.
\end{claim}
\begin{proof}
    Notice that for any $p,q \in \{1,\ldots, d\}$ if one defines $B_{p,q}(\bs x,\bs y) := (\bs\theta_p(\bs x)-\bs\theta_p(\bs y))(\bs\theta_q(\bs x)-\bs\theta_q(\bs y))$ then the Fourier transform of $B_{p,q}(\bs x,\bs y)$ with respect to $(\bs x,\bs y)$, evaluated at frequency vector $(\bs k,\bs k^\prime)$, is given by
    \begin{align*}
       B_{p,q,\bs k,\bs k^\prime} = \int \bs\theta_{p,\bs\omega}\Big(
        \bs\theta_{q,\bs k - \bs \omega}&\delta_{-\bs k^\prime}
        + \bs\theta_{q,\bs k^\prime - \bs \omega}\delta_{-\bs k}\delta_{\bs\omega+\bs\omega^\prime-\bs k^\prime} \\
        &- \bs\theta_{q,\bs k^\prime}\delta_{\bs\omega-\bs k}\delta_{\bs\omega^\prime-\bs k^\prime}
        - \bs\theta_{q,\bs k}\delta_{\bs\omega^\prime-\bs k}\delta_{\bs\omega-\bs k^\prime}
        \Big)d\bs\omega.
    \end{align*}
    Using the fact that  Fourier transform of $\bs C^{(2)}_{p,q}(\bs x-\bs y)$ equals ${(2\pi)}^{d/2}\bs C^{(2)}_{p,q,\bs k} \delta_{\bs k+\bs k^\prime}$ one obtains
    \begin{align}
        \mathcal O(\phi^2)_{\bs k, \bs k^\prime} &= \sum_{p,q=1}^d\Bigl[
        \,
        \underbrace{\bs C^{(2)}_{p,q}(\bs x-\bs y)}_{A_{p,q}(\bs x,\bs y)}
        \underbrace{(\bs\theta_p(\bs x)-\bs\theta_p(\bs y))(\bs\theta_q(\bs x)-\bs\theta_q(\bs y))}_{B_{p,q}(\bs x,\bs y)}
        \,
        \Bigr]{\vphantom{\int}}_{\bs k,\bs k^\prime} \nonumber\\
        &=\sum_{p,q=1}^d
        \iint  A_{p,q,\bs z+\bs k, \bs z^\prime+\bs k^\prime}B_{p,q,-\bs z, -\bs z^\prime}\frac{d\bs z \, d\bs z^\prime}{{(2\pi)}^d} \nonumber\\
        &=\sum_{p,q=1}^d\int
        \bs\theta_{p,\bs\omega}\bs\theta_{q,\bs k+\bs k^\prime - \bs\omega}
		 \Bigl(
         {\bs C^{(2)}_{p,q,-\bs k^\prime}}
        + {\bs C^{(2)}_{p,q,\bs k}}
        - {\bs C^{(2)}_{p,q,\bs k - \bs\omega}}
        - {\bs C^{(2)}_{p,q,\bs\omega-\bs k^\prime}}
        \Bigr)\frac{d\bs\omega}{(2\pi)^{d/2}}.
        \label{2p349571 first order bias behavior}
    \end{align}
    Making the substitution $\bs k^\prime \rightarrow -\bs k$ and $\bs k\rightarrow \bs k+\bs \ell$ in (\ref{2p349571 first order bias behavior}) proves (\ref{eq: the form of Op2}). Equation (\ref{eq: the form of hatphibias}) immediately follows from the fact that
    \begin{align*}
	\hat\phi_{\bs \ell}\Big\{{\bs C^{(2)}_{p,q,\bs k+\bs \ell}}- {\bs C^{(2)}_{p,q,\bs k+\bs \ell - \bs\omega}}\Big\}
    &=\hat\phi_{\bs \ell}\Big\{{\bs C^{(2)}_{p,q,\bs k}}- {\bs C^{(2)}_{p,q,\bs k + \bs\omega}}\Big\}
    \end{align*}
    which is established by utilizing the three properties $\bs C^{(2)}_{-\bs k} = \bs C^{(2)}_{\bs k}\in\Bbb R^{d\times d}$, $\bs C^{(1)^*}_{-\bs k} = \bs C^{(1)}_{\bs k}$ and $\bs C^{(1)}_{-\bs k} = -\bs C^{(1)}_{\bs k}$ along with the change of variables $\tilde {\bs k} = - \bs k - \bs \ell $. Finally, using Wick's theorem for Gaussian $\bs \theta(\bs x)$ gives
    \begin{align}
    \label{eq: wicks expansion of theta}
    E\big(\bs \theta_{p,\bs\omega}\bs \theta_{q,\bs \ell - \bs\omega}\bs \theta_{p^\prime,\bs\omega^\prime}^* \bs \theta_{q^\prime,\bs\ell^\prime - \bs\omega^\prime}^*\big)
    &= \delta_{\bs \ell - \bs \ell^\prime}\big(C^{\bs \theta\bs \theta}_{p,p^\prime,\bs \omega}C^{\bs \theta\bs \theta}_{q,q^\prime,\bs \ell - \bs \omega}\delta_{\bs \ell - \bs\omega^\prime - \bs \omega}
        + C^{\bs \theta\bs \theta}_{p,q^\prime,\bs \omega}C^{\bs \theta\bs \theta}_{q,p^\prime,\bs \ell - \bs \omega}\delta_{ \bs\omega^\prime + \bs \omega}  \big)
    \end{align}
    when $\bs \ell \neq 0$ or $\bs \ell^\prime \neq 0$. Expanding the quadratic $E\big(\hat\phi_{\bs \ell}^{\text{bias}} \hat\phi_{\bs \ell^\prime}^{\text{bias}^*}\big)$, applying Fubini and (\ref{eq: wicks expansion of theta}) then gives (\ref{eq: the form of the spectrum for hatphibias}) as was to be shown.
\end{proof}

\end{document}